\documentclass[reqno,12pt,a4paper]{amsart}
\usepackage[T1]{fontenc}
\usepackage[utf8]{inputenc}
\usepackage[english]{babel}
\usepackage{amsmath,amssymb,amsthm}
\usepackage{bm}
\usepackage[margin=25mm]{geometry}
\usepackage[hidelinks]{hyperref}
\newcommand{\cN}{\mathbb N}
\newcommand{\cZ}{\mathbb Z}
\newcommand{\PP}{\mathbb P}

\theoremstyle{plain}
\newtheorem{theorem}{Theorem}[section]
\newtheorem{lemma}{Lemma}[section]
\newtheorem{corollary}{Corollary}[section]
\theoremstyle{definition}

\theoremstyle{remark}
\newtheorem{remark}{Remark}
\numberwithin{equation}{section}

\newcommand{\cR}{\mathcal R}
\newcommand{\E}{\mathbb E}
\newcommand{\ind}{\mathbf 1}
\allowdisplaybreaks[2]

\title[Consecutive composite polynomial values]{Long intervals of consecutive\protect\\ composite values of polynomials}
\author{Artyom Radomskii}
\address{HSE University, Moscow, Russian Federation}
\email{artyom.radomskii@mail.ru}
\date{}

\begin{document}
\begin{abstract}
Let $f\in\mathbb Q[t]$ be an irreducible integer-valued polynomial
of positive degree with positive leading coefficient. We show that
there is a constant $c_f>0$ such that, for every sufficiently large
real $X$, the interval $(X/2,X]$ contains a block of at least
$c_f\log X\log\log X$ consecutive integers $n$ for which
$f(n)$ is composite. The exponent of $\log\log X$ does not depend on
the degree of $f$. This improves a result of Ford and Gabdullin,
replacing a small fixed power of $\log\log X$ by
$\log\log X$. The result includes the case $f(t)=t^2+1$.
\end{abstract}
\maketitle

\section{Introduction}

For an integer-valued polynomial $f$, write
\[
 \Lambda_f=\{n\in\cN:f(n)\text{ is prime}\}.
\]
We study long intervals of consecutive integers disjoint from
$\Lambda_f$, and obtain lower bounds for the length of a block on
which all values of $f$ are composite.

For comparison, recall the conjectures on prime values of polynomials
with integer coefficients. Let $f\in\cZ[t]$ be irreducible over $\mathbb Q$, of degree
$m\ge2$, and with positive leading coefficient. Suppose that no prime
divides all the values $f(n)$, $n\in\cZ$. Equivalently, writing
\[
 \nu_f(p)=\#\{a\in\cZ/p\cZ:f(a)\equiv0\pmod p\},
\]
we require $\nu_f(p)<p$ for every prime $p$. Bouniakowsky's
conjecture~\cite{Bou} asserts that $\Lambda_f$ is infinite under these
hypotheses. No individual polynomial satisfying these hypotheses is
known to take infinitely many prime values; even the case
$f(t)=t^2+1$ remains open. For linear polynomials
with positive leading coefficient and no fixed prime divisor, the
corresponding infinitude follows from Dirichlet's theorem on primes in
arithmetic progressions.

The case $f(t)=t^2+1$ is usually called Landau's conjecture:
there should be infinitely many positive integers $n$ for which
$n^2+1$ is prime. Landau listed this question among four problems
on prime numbers in his 1912 address to the International Congress
of Mathematicians in Cambridge; see~\cite[p.~106]{Landau1912}.
Apart from $1^2+1=2$, prime values can occur only at even arguments:
$n^2+1$ is even and greater than $2$ for every odd $n>1$.

The conjecture of Bateman and Horn~\cite{Bateman.Horn,Bateman.Horn2}
predicts an asymptotic formula for the number of simultaneous prime
values of a finite collection of irreducible polynomials. For a single
polynomial satisfying the above hypotheses, the prediction is
\[
 \#\bigl(\Lambda_f\cap[1,X]\bigr)
 \sim \mathfrak S(f)\frac{X}{\log X},
 \qquad X\to\infty,
\]
where
\[
 \mathfrak S(f)
 =\frac1m\lim_{z\to\infty}
   \prod_{p\le z}\frac{1-\nu_f(p)/p}{1-1/p}>0.
\]
The local factors account for the residue classes on which $f$ is
not divisible by $p$, and the factor $1/m$ comes from
$\log f(n)\sim m\log n$. The conjecture therefore predicts an
average spacing of order $\log X$ between elements of $\Lambda_f$
in $(X/2,X]$. The intervals in Theorem~\ref{thm:main}
exceed this scale by a factor of order $\log\log X$.

We allow integer-valued polynomials with rational coefficients.
By the classical characterization of P\'olya~\cite{Polya}, a polynomial
$f\in\mathbb Q[t]$ of degree $m$ satisfies $f(\cZ)\subseteq\cZ$ if and
only if it has an expansion
\[
 f(t)=\sum_{j=0}^{m}a_j\binom tj,
 \qquad a_j\in\cZ,
\]
where
\[
 \binom t0=1,\qquad
 \binom tj=\frac{t(t-1)\cdots(t-j+1)}{j!}\quad(j\ge1).
\]
In particular, $m!f\in\cZ[t]$. After excluding the finitely many
primes dividing a common denominator, divisibility of the values of
$f$ is described by ordinary polynomial congruences. We can thus use
the same sieve as for polynomials in $\cZ[t]$.

A general approach to this problem was developed by Ford, Konyagin,
Maynard, Pomerance, and Tao \cite{FKMPT}. For each prime $p$, let $I_p\subseteq\cZ/p\cZ$
be a set of forbidden residue classes. Assuming that $|I_p|<p$, that
the cardinalities $|I_p|$ are bounded, that
\[
 \prod_{p\le x}\left(1-\frac{|I_p|}{p}\right)
 \sim\frac{c_I}{\log x}
 \qquad(c_I>0),
\]
and that the primes with $I_p\ne\varnothing$ have a positive relative
density, their theorem gives gaps of length at least $x(\log x)^\delta$
in the set
\[
 \{n\in\cZ:n\bmod p\notin I_p\text{ for every }p\le x\},
\]
for some $\delta>0$ depending only on that density. Taking polynomial root sets and omitting finitely many exceptional
primes gives an application to polynomial values.

To state the resulting polynomial bound, define, in the corrected
normalization of~\cite{FKMPT.corr},
\[
 C(\eta)=\sup\left\{
 \delta\in(0,1/2):
 \frac{6\cdot10^{2\delta}}{\log(1/(2\delta))}<\eta
 \right\}
 \qquad(\eta>0).
\]
For an irreducible integer-valued polynomial $f$ of degree $m\ge1$
with positive leading coefficient, the polynomial application
of~\cite{FKMPT,FKMPT.corr} gives, for every fixed
$0<\delta<C(1/m)$ and all sufficiently large $X$, an interval of
consecutive integers in $(X/2,X]$ of length at least
\[
 (\log X)(\log\log X)^\delta
\]
on which $f$ is composite. Here the admissible exponent depends on the
degree, and $C(1/m)$ decreases exponentially as $m$ grows.

Ford and Gabdullin~\cite{FG} subsequently proved the same assertion
for every $0<\delta<C(1)$ and every degree $m$. Since
$C(1)=1/834.109\ldots$, one may take $\delta=1/835$.
Their argument uses all the root classes at each available prime and
exploits additional properties of polynomial sieving systems.
We replace this small fixed power of $\log\log X$ by
$\log\log X$.

\begin{theorem}\label{thm:main}
Let $f\in\mathbb Q[t]$ be a polynomial of positive degree, irreducible
over $\mathbb Q$, with positive leading coefficient, and suppose that
$f(\cZ)\subseteq\cZ$. There exist constants $c_f>0$ and $X_0(f)>0$
such that, for every real $X\ge X_0(f)$, there are integers $N$ and
$L\ge1$ satisfying
\begin{equation}\label{eq:main}
 \{N+1,\ldots,N+L\}\subset (X/2,X],\qquad
 L\ge c_f\log X\log\log X,
\end{equation}
for which all the numbers $f(N+1),\ldots,f(N+L)$ are composite.
\end{theorem}

Theorem~\ref{thm:main} does not require the absence of a fixed prime
divisor. If a prime divides every value of $f$, all sufficiently large
positive values are composite, so the assertion is immediate.
The substantive case is when there is no fixed prime divisor.
Both $c_f$ and $X_0(f)$ may depend on the polynomial; the exponent
of $\log\log X$ in~\eqref{eq:main} is $1$ in every degree.

For Landau's polynomial, the theorem gives the following corollary.

\begin{corollary}\label{cor:squares-plus-one}
There exist absolute constants $c>0$ and $X_0>0$ such that, for every
real $X\ge X_0$, there are integers $N$ and $L\ge1$ satisfying
\[
 \{N+1,\ldots,N+L\}\subset(X/2,X],
 \qquad L\ge c\log X\log\log X,
\]
for which all the numbers
\[
 (N+1)^2+1,\ (N+2)^2+1,\ \ldots,\ (N+L)^2+1
\]
are composite.
\end{corollary}

\begin{proof}
The polynomial $f(t)=t^2+1$ belongs to $\cZ[t]$, has positive
leading coefficient, and has no rational root. Being quadratic,
it is irreducible over $\mathbb Q$, so Theorem~\ref{thm:main} applies.
\end{proof}

Thus the set $\{n\in\cN:n^2+1\text{ is prime}\}$ has gaps of length
at least $c\log X\log\log X$ in every sufficiently large interval
$(X/2,X]$. This conclusion is unconditional and does not require the
infinitude asserted by Landau's conjecture.

We adapt the short-translate construction of \cite{LG} to polynomial
values. The product cutoff is analogous to the simplex restriction
used in the Maynard--Tao sieve; see Maynard~\cite{Maynard}.
For polynomial values, we need different local factors and must count
the residue classes in each congruence condition when passing from
complete averages to averages over a finite interval.

In the construction of \cite{LG}, each auxiliary prime excludes a
single residue class for each linear form. For polynomial values, the
number of excluded classes varies with the prime and may be zero or
larger than one. After clearing denominators and omitting finitely
many exceptional primes, this number is the root count $\rho(p)$
used below. The local factors and the coefficients of the divisor
sums must therefore be adapted to these root counts.
Lemmas~\ref{lem:mertens} and~\ref{lem:summation} provide the
polynomial Mertens estimates and the uniform divisor-sum estimate
needed to control the resulting coefficients as the sieve parameters
grow.

There is also a separation issue. For distinct shifts of a linear
form, the excluded classes are disjoint as soon as the auxiliary
prime exceeds the distance between the shifts. For polynomial
shifts, this inequality alone does not ensure disjointness.
Writing $g$ for a fixed nonzero integer multiple of $f$ with integer
coefficients, a common root of $g(T+s_i)$ and $g(T+s_j)$ modulo $p$
forces $p$ to divide their resultant. Irreducibility ensures that
this resultant is nonzero when $s_i\ne s_j$.
Lemma~\ref{lem:separation} bounds these resultants uniformly over
the shifts in the interval and places all their prime divisors below
the auxiliary range. The resulting disjointness gives the local
moment identities used to estimate the diagonal and off-diagonal
contributions of the truncated weight.

Finally, the passage to a finite interval of translates requires
keeping track of the number of residue classes. For linear forms,
a compatible collection of divisibility conditions gives one class
modulo the product of the primes involved. Here it gives
$\prod_{p\mid q}\rho(p)$ classes modulo the corresponding squarefree
modulus $q$, so the endpoint errors must be multiplied by this
number. In Section~\ref{subsec:finite-averages}, we control this
factor together with the sum of the absolute values of the
coefficients in the indicator expansion. These estimates are
uniform when the number of remaining positions grows up to
$\eta_f x$, for a sufficiently small constant $\eta_f>0$.
Together with the coefficient and moment estimates, they extend
the short-translate construction to the polynomial setting in
Lemma~\ref{lem:translate}.

We first average over residue classes modulo the product of the small
primes. This leaves $O_f(x)$ uncovered positions in an interval of
length comparable to $x\log x$. A truncated product of divisor sums
then gives a parameter $t\le e^x$ for which each remaining position
has an auxiliary prime divisor. The truncation is what permits
averaging over this range of $t$: the full period of the divisibility
conditions is much larger.

\section{Notation}\label{sec:notation}

Let $\mathbb P$ denote the set of prime numbers. We write $\cN=\{1,2,\ldots\}$, $\cZ$, and $\mathbb Q$ for the positive
integers, the integers, and the rational numbers, respectively.
The letter $p$ denotes a prime; sums and products indexed by $p$
run over primes in the indicated range.
We denote the cardinality of a finite set $S$ by $|S|$ or $\#S$.
For integers $a,b$, $(a,b)$ denotes their nonnegative greatest common
divisor, while $\operatorname{lcm}(a,b)$ denotes their least common
multiple when $a,b$ are positive. We write $d\mid n$ for divisibility.
The function $\omega(n)$ counts the distinct prime divisors of a positive
integer $n$, with $\omega(1)=0$. A positive integer is
\emph{squarefree} if it is divisible by no square of a prime.
The M\"obius function $\mu$ is defined by $\mu(1)=1$,
$\mu(n)=(-1)^{\omega(n)}$ for squarefree $n$, and $\mu(n)=0$
otherwise.

All logarithms are natural, and $\gamma_E$ denotes Euler's constant.
For a real number $a$, we put $a_+=\max\{a,0\}$; the symbols
$\lfloor a\rfloor$ and $\lceil a\rceil$ denote its floor and ceiling.
The indicator $\ind_{\mathcal E}$ equals $1$ if the condition
$\mathcal E$ holds and $0$ otherwise. Empty sums are $0$, and empty
products are $1$. The length of an
interval of consecutive integers means the number of integers it contains.

For a nonnegative quantity $V$, the relations $U=O_f(V)$ and
$U\ll_f V$ mean that $|U|\le C V$ for a constant $C$ depending only
on $f$ and on fixed parameters chosen in terms of $f$.
For positive quantities, $U\asymp_f V$ means that both
$U\ll_f V$ and $V\ll_f U$ hold. The relation $U\sim V$ means
$U/V\to1$, and $o_f(1)$ denotes a quantity tending to zero with $f$
fixed. Unless another limiting process is specified, the growing
parameter is $x\to\infty$ (or $X\to\infty$ in the main theorem).
All assertions of uniformity apply to both constants and thresholds.
In the covering argument, these are independent of $H$, $S$, and
the integer $b$.

\section{Setup}
Fix $f$ as in Theorem~\ref{thm:main}, and put $m=\deg f$.
Choose an integer $D_f\ge1$ such that $g=D_ff\in\cZ[t]$.
The polynomial $g$ is still irreducible over $\mathbb Q$.
Choose a fixed $p_0$ larger than $2m$ and than every prime divisor of
$D_f$ and of the leading coefficient of $g$. For primes $p$, define
\[
 \rho(p)=
 \begin{cases}
 0,&p\le p_0,\\
 \#\{a\bmod p:g(a)\equiv0\pmod p\},&p>p_0.
 \end{cases}
\]
For $p>p_0$, reduction modulo $p$ preserves the degree of $g$, so
$0\le\rho(p)\le m<p$. Since $f(n)$ is an integer and $p\nmid D_f$,
\begin{equation}\label{eq:divisibility-equivalence}
 p\mid g(n)\quad\Longleftrightarrow\quad p\mid f(n)
 \qquad(n\in\cZ,\ p>p_0).
\end{equation}
Write
\[
 Q(v)=\prod_{p\le v}p,\qquad
 \sigma_f(v)=\prod_{p\le v}\left(1-\frac{\rho(p)}p\right).
\]

For a function $U$ on the residue classes modulo a prime $p$, we write
\[
 \E_p U=\frac1p\sum_{t=0}^{p-1}U(t).
\]
When $P$ is the squarefree product of auxiliary primes introduced below,
$\E U$ denotes the complete average of a $P$-periodic function:
\[
 \E U=\frac1P\sum_{t=0}^{P-1}U(t).
\]
This is distinct from the finite average $T^{-1}\sum_{t=1}^T U(t)$.

\section{Arithmetic preliminaries}

\begin{lemma}[Polynomial Mertens estimates]\label{lem:mertens}
For $v\ge2$,
\begin{equation}\label{eq:weighted-mertens}
 \sum_{p\le v}\frac{\rho(p)\log p}{p}=\log v+O_f(1).
\end{equation}
There is a constant $\mathfrak c_f>0$ such that
\begin{equation}\label{eq:polynomial-mertens}
 \sigma_f(v)=\frac{\mathfrak c_f}{\log v}
 \left(1+O_f\left(\frac1{\log v}\right)\right)
 \qquad(v\longrightarrow\infty).
\end{equation}
\end{lemma}

\begin{proof}
Let $g_0\in\mathbb Z[T]$ be the primitive part of $g$, obtained
by dividing $g$ by the positive greatest common divisor of its
coefficients. Since $g$ is irreducible over $\mathbb Q$, the
polynomial $g_0$ is irreducible in $\mathbb Z[T]$. For each prime
$p$, put
\[
\nu(p)=\#\{a\in\mathbb Z/p\mathbb Z:g_0(a)\equiv0\pmod p\}.
\]
The prime ideal theorem gives the following estimate; see
Landau~\cite{Landau1903} and, for the root-counting formulation,
\cite[Theorem~3.2.1, p.~36]{CM}. For some constant $c>0$ depending
only on $g_0$, we have
\[
\sum_{p\le v}\nu(p)
=\operatorname{li}(v)
 +O_f\!\left(v e^{-c\sqrt{\log v}}\right),
\qquad
\operatorname{li}(v)=\int_2^v\frac{dt}{\log t}.
\]
Every prime dividing the content of $g$ divides its leading
coefficient and is therefore at most $p_0$. Thus, for $p>p_0$,
the reductions of $g$ and $g_0$ have the same roots and
$\rho(p)=\nu(p)$. The convention $\rho(p)=0$ for $p\le p_0$
changes the sum by $O_f(1)$. Enlarging the implied constant,
we obtain
\[
B(v):=\sum_{p\le v}\rho(p)
=\operatorname{li}(v)+R_f(v),
\qquad
R_f(v)=O_f\!\left(v e^{-c\sqrt{\log v}}\right)
\qquad(v\ge2).
\]

Partial summation with the weight $(\log t)/t$ gives
\[
\sum_{p\le v}\frac{\rho(p)\log p}{p}
=\frac{B(v)\log v}{v}
 +\int_2^v B(t)\frac{\log t-1}{t^2}\,dt.
\]
For the contribution of $\operatorname{li}(t)$, use
$\operatorname{li}(2)=0$ and
$\operatorname{li}'(t)=1/\log t$. Integration by parts yields
\[
\frac{\operatorname{li}(v)\log v}{v}
+\int_2^v\operatorname{li}(t)\frac{\log t-1}{t^2}\,dt
=\int_2^v\frac{dt}{t}
=\log v-\log2.
\]
Hence
\[
\sum_{p\le v}\frac{\rho(p)\log p}{p}
=\log v-\log2
 +\frac{R_f(v)\log v}{v}
 +\int_2^v R_f(t)\frac{\log t-1}{t^2}\,dt.
\]
The endpoint error is bounded, since
\[
\frac{R_f(v)\log v}{v}
=O_f\!\left((\log v)e^{-c\sqrt{\log v}}\right)
=O_f(1).
\]
The integral error is bounded uniformly in $v$ as well:
\[
\int_2^v
 \left|R_f(t)\frac{\log t-1}{t^2}\right|\,dt
\ll_f
\int_2^\infty e^{-c\sqrt{\log t}}\frac{\log t+1}{t}\,dt
<\infty.
\]
Indeed, the substitution $u=\sqrt{\log t}$ transforms the last
integral into
\[
2\int_{\sqrt{\log2}}^\infty e^{-cu}(u^3+u)\,du,
\]
which converges. Therefore
\[
\sum_{p\le v}\frac{\rho(p)\log p}{p}
=\log v+O_f(1),
\]
proving \eqref{eq:weighted-mertens}.

For the product estimate, set
\[
 W(v)=\sum_{p\le v}\frac{\rho(p)\log p}{p}
     =\log v+E_f(v),
 \qquad E_f(v)=O_f(1).
\]
Partial summation gives
\[
 \begin{aligned}
 \sum_{p\le v}\frac{\rho(p)}p
 &=
 \frac{W(v)}{\log v}
 +\int_2^v\frac{W(t)}{t(\log t)^2}\,dt\\
 &=
 1+\frac{E_f(v)}{\log v}
 +\log\log v-\log\log2
 +\int_2^v\frac{E_f(t)}{t(\log t)^2}\,dt.
 \end{aligned}
\]
Since $E_f(t)$ is bounded, the integral
\[
 \int_2^\infty\frac{E_f(t)}{t(\log t)^2}\,dt
\]
converges absolutely. Furthermore,
\[
 \int_v^\infty\frac{dt}{t(\log t)^2}
 =\frac1{\log v}.
\]
Hence, on defining
\[
 c'_f
   =1-\log\log2
    +\int_2^\infty\frac{E_f(t)}{t(\log t)^2}\,dt,
\]
we obtain
\[
 \sum_{p\le v}\frac{\rho(p)}p
 =\log\log v+c'_f+O_f(1/\log v).
\]

By the choice of $p_0$, we have $0\le\rho(p)/p<1/2$ for every
prime $p$. Consequently,
\[
 \varepsilon_p
   :=\log\left(1-\frac{\rho(p)}p\right)+\frac{\rho(p)}p
   =O_f(p^{-2}).
\]
Thus $\sum_p\varepsilon_p$ converges absolutely, and
\[
 \sum_{p>v}|\varepsilon_p|
 \ll_f\sum_{n>v}\frac1{n^2}
 \ll_f\frac1v.
\]
Taking logarithms in the definition of $\sigma_f(v)$, we find
\[
 \begin{aligned}
 \log\sigma_f(v)
 &=
 \sum_{p\le v}\log\left(1-\frac{\rho(p)}p\right)\\
 &=
 -\sum_{p\le v}\frac{\rho(p)}p
 +\sum_{p\le v}\varepsilon_p\\
 &=
 -\log\log v+c''_f+O_f(1/\log v),
 \end{aligned}
\]
where
\[
 c''_f=-c'_f+\sum_p\varepsilon_p.
\]
Exponentiating and using $e^u=1+O(u)$ as $u\to0$, we obtain
\[
 \sigma_f(v)
 =\frac{e^{c''_f}}{\log v}
   \left(1+O_f\left(\frac1{\log v}\right)\right).
\]
This proves \eqref{eq:polynomial-mertens}, with
$\mathfrak c_f=e^{c''_f}>0$.
\end{proof}

\begin{lemma}[Separation of shifted root sets]\label{lem:separation}
Put $\alpha_f=m^2+6$. For all sufficiently large $x$, every
$1\le H\le x(\log x)^2$, and every prime $p>x^{\alpha_f}$,
the root sets modulo $p$ of the polynomials $g(u+s)$, as $s$ ranges
over distinct integers in $[1,H]$, are pairwise disjoint.
\end{lemma}

\begin{proof}
We shall use the following properties of the resultant.

Let $P,Q\in R[T]$ be polynomials of positive degrees $r,s$,
respectively, where $R$ is an integral domain.
The Sylvester matrix $S(P,Q)$ is the $(r+s)\times(r+s)$ matrix
whose rows are the coefficient vectors of
\[
T^{s-1}P,\ T^{s-2}P,\ \ldots,\ P,
\qquad
T^{r-1}Q,\ T^{r-2}Q,\ \ldots,\ Q,
\]
with the coefficients arranged according to the descending
powers
\[
T^{r+s-1},\ T^{r+s-2},\ \ldots,\ T,\ 1.
\]
The resultant of $P$ and $Q$ with respect to $T$ is defined by
\[
\operatorname{Res}_T(P,Q)=\det S(P,Q).
\]
Since a determinant is a polynomial with integer coefficients
in the matrix entries, the resultant belongs to $R$.
In particular, it is an integer when $R=\mathbb Z$, and it
belongs to $\mathbb Z[h]$ when the coefficients of $P$ and $Q$
belong to $\mathbb Z[h]$.

Over a field, write
\[
P(T)=a\prod_{i=1}^{r}(T-\alpha_i),
\qquad
Q(T)=b\prod_{j=1}^{s}(T-\beta_j),
\]
where the roots are taken in an algebraic closure and are
listed with multiplicity. The root formula for the resultant is
\[
\operatorname{Res}_T(P,Q)
=a^s\prod_{i=1}^{r}Q(\alpha_i)
=a^s b^r
 \prod_{i=1}^{r}\prod_{j=1}^{s}(\alpha_i-\beta_j).
\]
Consequently, the resultant vanishes if and only if $P$ and $Q$
have a common root in an algebraic closure, or equivalently,
a common factor of positive degree over the coefficient field.

The resultant satisfies the identity
\[
U(T)P(T)+V(T)Q(T)=\operatorname{Res}_T(P,Q),
\qquad U,V\in R[T].
\]
To prove this over the coefficient ring, put $N=r+s$ and write
\[
S=S(P,Q)=(s_{ik})_{1\le i,k\le N}.
\]
Recall that its rows are the coefficient vectors of
\[
T^{s-1}P,\ldots,P,\qquad T^{r-1}Q,\ldots,Q,
\]
with coefficients arranged according to
$T^{N-1},T^{N-2},\ldots,T,1$.

For each $1\le i\le N$, let $S^{(i,N)}$ be the matrix obtained
from $S$ by deleting its $i$th row and its last column.
Define
\[
\gamma_i=(-1)^{i+N}\det S^{(i,N)}.
\]
Thus $\gamma_i$ is the cofactor of the entry $s_{iN}$.
All the $\gamma_i$ belong to $R$, since they are determinants
of matrices with entries in $R$.

We claim that
\[
\sum_{i=1}^{N}\gamma_i s_{ik}
=
\begin{cases}
0,&1\le k<N,\\
\det S,&k=N.
\end{cases}
\]
For $k=N$, this is precisely the expansion of $\det S$
along its last column.

For $k<N$, let $S_k$ be the matrix obtained from $S$ by
replacing its last column with its $k$th column. The matrix
$S_k$ has two equal columns, so $\det S_k=0$.
Moreover, the cofactors of the entries in its last column
are still $\gamma_1,\ldots,\gamma_N$: deleting that column
removes the only column that was changed. Expanding
$\det S_k$ along its last column therefore gives
\[
0=\det S_k=\sum_{i=1}^{N}\gamma_i s_{ik}.
\]
This proves the claim. Equivalently,
\[
(\gamma_1,\ldots,\gamma_N)S
=(0,\ldots,0,\det S).
\]

Introduce the column vector
\[
\boldsymbol z(T)
=
\begin{pmatrix}
T^{N-1}\\
T^{N-2}\\
\vdots\\
T\\
1
\end{pmatrix}.
\]
By the definition of the rows of $S$, we have
\[
S\boldsymbol z(T)
=
\begin{pmatrix}
T^{s-1}P(T)\\
\vdots\\
P(T)\\
T^{r-1}Q(T)\\
\vdots\\
Q(T)
\end{pmatrix}.
\]
Multiplying the preceding row identity by $\boldsymbol z(T)$,
we obtain
\[
\sum_{i=1}^{s}\gamma_i T^{s-i}P(T)
+
\sum_{j=1}^{r}\gamma_{s+j}T^{r-j}Q(T)
=\det S.
\]

Consequently, the explicit choices
\[
U(T)=\sum_{i=1}^{s}\gamma_i T^{s-i},
\qquad
V(T)=\sum_{j=1}^{r}\gamma_{s+j}T^{r-j}
\]
satisfy
\[
U(T)P(T)+V(T)Q(T)
=\det S
=\operatorname{Res}_T(P,Q).
\]
The formulas also show that $\deg U<s$ and $\deg V<r$
whenever these polynomials are nonzero.

No division in $R$ is needed, so the identity also holds when
$\det S=0$. In particular, if $P,Q\in\mathbb Z[T]$, then
$U,V\in\mathbb Z[T]$. If the coefficients of $P,Q$ belong
to $\mathbb Z[h]$, the same argument gives
$U,V\in\mathbb Z[h][T]=\mathbb Z[T,h]$.

Apply these identities to
\[
g(T)=a_mT^m+\cdots+a_0,\qquad a_m\ne0.
\]
Define
\[
\mathcal V(h)
=\operatorname{Res}_T\bigl(g(T),g(T+h)\bigr).
\]
The coefficients of $g(T+h)$, viewed as a polynomial in $T$,
belong to $\mathbb Z[h]$. Hence
\[
\mathcal V(h)\in\mathbb Z[h].
\]
If $\alpha_1,\ldots,\alpha_m$ are the complex roots of $g$,
the roots of $g(T+h)$ are $\alpha_1-h,\ldots,\alpha_m-h$.
Both polynomials have leading coefficient $a_m$, so the
root formula gives
\[
\mathcal V(h)
=a_m^{2m}
 \prod_{i=1}^{m}\prod_{j=1}^{m}
 (h+\alpha_i-\alpha_j).
\]
Thus $\mathcal V$ has degree at most $m^2$ in $h$, and,
since $g$ is fixed in terms of $f$,
\begin{equation}\label{eq:resultant-size}
|\mathcal V(h)|\ll_f(1+|h|)^{m^2}.
\end{equation}

We next show that
\[
\mathcal V(h)\ne0
\qquad(h\in\mathbb Z\setminus\{0\}).
\]
Suppose otherwise. Then $g(T)$ and $g(T+h)$ have a common
factor of positive degree over $\mathbb Q$.
Since $g(T)$ is irreducible over $\mathbb Q$, it must divide
$g(T+h)$. These polynomials have the same degree and the
same leading coefficient, so they must be equal.
However, the coefficient of $T^{m-1}$ in their difference
$g(T+h)-g(T)$ is
\[
ma_mh,
\]
which is nonzero. This contradiction proves the claim.

The resultant identity over the coefficient ring $\mathbb Z[h]$
provides polynomials $A,B\in\mathbb Z[T,h]$ such that
\[
A(T,h)g(T)+B(T,h)g(T+h)=\mathcal V(h).
\]
Therefore, for any integers $u_0,h$ and any prime $p$,
\[
p\mid g(u_0),\qquad p\mid g(u_0+h)
\quad\Longrightarrow\quad
p\mid\mathcal V(h).
\]
Indeed, evaluating the identity at $(T,h)=(u_0,h)$ gives
\[
\mathcal V(h)
=A(u_0,h)g(u_0)+B(u_0,h)g(u_0+h),
\]
whose right-hand side is divisible by $p$, since both
$A(u_0,h)$ and $B(u_0,h)$ are integers.

Suppose now that, for distinct integers $s_i,s_j\in[1,H]$,
the polynomials $g(T+s_i)$ and $g(T+s_j)$ have a common
root modulo $p$. Choose an integer representative $t_0$
of this root. Then
\[
p\mid g(t_0+s_i),
\qquad
p\mid g(t_0+s_j).
\]
Put
\[
u_0=t_0+s_i,
\qquad
h=s_j-s_i.
\]
We have $h\ne0$, $|h|\le H-1$, and
\[
t_0+s_j=u_0+h.
\]
The preceding divisibility implication therefore yields
\[
p\mid\mathcal V(s_j-s_i).
\]

On the other hand, \eqref{eq:resultant-size} and
$H\le x(\log x)^2$ give
\[
0<|\mathcal V(s_j-s_i)|
\ll_f(1+H)^{m^2}
\ll_f x^{m^2}(\log x)^{2m^2}< x^{m^{2}+6}
\]
for all sufficiently large $x$.
Recalling that $\alpha_f=m^2+6$ and $p>x^{\alpha_f}$,
we obtain
\[
 0<|\mathcal V(s_j-s_i)|<p.
\]
This contradicts $p\mid\mathcal V(s_j-s_i)$. Hence the shifted
root sets are pairwise disjoint, as required.
\end{proof}

For $z>p_0$, let $h_z$ be the multiplicative function supported on
squarefree integers given by
\[
 h_z(1)=1,\qquad
 h_z(p)=
 \begin{cases}
 \rho(p)/(p-\rho(p)),&p>z,\\
 0,&p\le z.
 \end{cases}
\]
Thus $h_z(n)=0$ if $n$ is not squarefree or has a prime divisor at
most $z$.

\begin{lemma}[A uniform divisor-sum estimate]\label{lem:summation}
Uniformly for $v\ge2$ and all sufficiently large $z$,
\begin{equation}\label{eq:summation}
 M_z(v):=\sum_{r\le v}h_z(r)=V_f(z)\log v+O_f(1),
\end{equation}
where, with $\gamma_E$ denoting Euler's constant,
\begin{equation}\label{eq:V}
 V_f(z)=\frac{e^{-\gamma_E}}{\mathfrak c_f}\sigma_f(z)
 =\frac{e^{-\gamma_E}}{\log z}
       \left(1+O_f\left(\frac1{\log z}\right)\right).
\end{equation}
\end{lemma}

\begin{proof}
We use the dimension-one case of \cite[Theorem~4.4]{Ford},
in the following form.
Let $q$ be a multiplicative function supported on squarefree
integers, and suppose that there are constants $A_1,A_2,L>0$
such that
\[
 0\le q(p)\le 1-A_1
 \qquad\text{for every prime }p,
\]
and
\[
 -L\le
 \sum_{w\le p\le y}q(p)\log p-\log(y/w)
 \le A_2
 \qquad(2\le w\le y).
\]
Let $h$ be the multiplicative function supported on squarefree
integers with
\[
 h(p)=\frac{q(p)}{1-q(p)}.
\]
Then, uniformly for $v\ge2$,
\[
 \sum_{r\le v}h(r)
 =
 \mathfrak S(q)\log v
 +O_{A_1,A_2}\bigl(\mathfrak S(q)(L+1)\bigr),
\]
where
\[
 \mathfrak S(q)
 =
 \prod_p(1-q(p))^{-1}(1-p^{-1}).
\]
In particular, the implied constant is independent of $L$.
In Ford's general statement, the two factors inside the error bound
simplify at sieve dimension $\kappa=1$ as follows:
\[
 (L+1)(\log v)^{\kappa-1}=(L+1)(\log v)^0=L+1,
 \qquad (L+1)^\kappa=L+1.
\]
Multiplication by $\mathfrak S(q)$ therefore gives
$2\mathfrak S(q)(L+1)$, and the factor $2$ is absorbed into
the implied constant.

Take $q_z(1)=1$ and
\[
 q_z(p)=
 \begin{cases}
  0,&p\le z,\\[2mm]
  \rho(p)/p,&p>z.
 \end{cases}
\]
Extend this definition multiplicatively to squarefree integers,
and set $q_z(r)=0$ whenever $r$ is divisible by the square of
a prime. Thus, for squarefree $r$,
\[
 q_z(r)=\prod_{p\mid r}q_z(p).
\]
The associated function $h$ in Ford's theorem is precisely $h_z$,
since
\[
 \frac{q_z(p)}{1-q_z(p)}
 =
 \begin{cases}
  0,&p\le z,\\[2mm]
  \displaystyle\frac{\rho(p)}{p-\rho(p)},&p>z.
 \end{cases}
\]

The dependence of the hypotheses on $z$ is as follows.
Since $0\le\rho(p)\le m$, for $z>2m$ we have
\[
 0\le q_z(p)\le\frac12
 \qquad\text{for every prime }p.
\]
Consequently, we may take $A_1=1/2$, independently of $z$.

For the second hypothesis, write
\[
 W(t):=\sum_{p\le t}\frac{\rho(p)\log p}{p}.
\]
By \eqref{eq:weighted-mertens},
\[
 W(t)=\log t+O_f(1)
 \qquad(t\ge2).
\]
For $2\le w\le y$, put
\[
 \Delta_z(w,y)
 :=
 \sum_{w\le p\le y}q_z(p)\log p-\log(y/w).
\]
There are three cases.

If $y\le z$, all terms in the prime sum vanish. Hence
\[
 \Delta_z(w,y)=-\log(y/w),
 \qquad
 -\log z\le\Delta_z(w,y)\le0.
\]

If $w\le z<y$, the nonzero terms are exactly those with
$z<p\le y$. Therefore
\begin{align*}
 \Delta_z(w,y)
 &=W(y)-W(z)-\log(y/w)\\
 &=\log(y/z)-\log(y/w)+O_f(1)\\
 &=-\log(z/w)+O_f(1).
\end{align*}
Since $2\le w\le z$, this gives
\[
 -\log z-O_f(1)
 \le\Delta_z(w,y)\le O_f(1).
\]

Finally, suppose that $z<w\le y$. In this case the cutoff
at $z$ does not affect the sum, so
\[
 \sum_{w\le p\le y}q_z(p)\log p
 =W(y)-W(w)+E(w),
\]
where
\[
 E(w):=\sum_{\substack{p\ \mathrm{prime}\\p=w}}
             \frac{\rho(p)\log p}{p}.
\]
The sum defining $E(w)$ has at most one term, and
\[
 0\le E(w)\le m\frac{\log w}{w}=O_f(1).
\]
It follows that
\[
 \Delta_z(w,y)=O_f(1).
\]

There is therefore a positive constant $C_f$, depending only on $f$,
such that
\begin{equation}\label{eq:ford-hypothesis}
 -\log z-C_f
 \le
 \sum_{w\le p\le y}q_z(p)\log p-\log(y/w)
 \le C_f
 \qquad(2\le w\le y).
\end{equation}
By \eqref{eq:ford-hypothesis}, the parameters in Ford's theorem can be chosen as
\[
 \kappa=1,\qquad
 A_1=\frac12,\qquad
 A_2=C_f,\qquad
 L=L_z:=\log z+C_f.
\]
The parameters $\kappa,A_1,A_2$ are independent of $z$.
Only $L_z$ varies with $z$, and its contribution is explicit
in the error term. Ford's theorem therefore yields
\[
 M_z(v)
 =
 \mathfrak S(q_z)\log v
 +O_f\bigl(\mathfrak S(q_z)(L_z+1)\bigr)
 \qquad(v\ge2),
\]
uniformly in $v$ and in sufficiently large $z$.
The factor $\mu^2(r)$ appearing in the sum in Ford's statement
may be omitted here, because $h_z(r)=0$ for nonsquarefree $r$.

We now evaluate the Euler product.
For $t>z$, the definition of $q_z$ gives
\[
 \prod_{p\le t}(1-q_z(p))
 =
 \prod_{z<p\le t}\left(1-\frac{\rho(p)}p\right)
 =
 \frac{\sigma_f(t)}{\sigma_f(z)}.
\]
Consequently,
\begin{align*}
 \mathfrak S(q_z)
 &=
 \lim_{t\to\infty}
 \prod_{p\le t}(1-q_z(p))^{-1}(1-p^{-1})\\
 &=
 \sigma_f(z)
 \lim_{t\to\infty}
 \frac{\displaystyle\prod_{p\le t}(1-p^{-1})}
      {\sigma_f(t)}.
\end{align*}
The ordinary Mertens product formula and
\eqref{eq:polynomial-mertens} give, respectively,
\[
 \prod_{p\le t}(1-p^{-1})
 =\frac{e^{-\gamma_E}}{\log t}(1+o(1)),
 \qquad
 \sigma_f(t)=\frac{\mathfrak c_f}{\log t}(1+o(1)).
\]
Hence
\[
 \mathfrak S(q_z)
 =\frac{e^{-\gamma_E}}{\mathfrak c_f}\sigma_f(z)
 =V_f(z).
\]
Applying \eqref{eq:polynomial-mertens} once more, now at $z$,
we obtain
\[
 V_f(z)
 =
 \frac{e^{-\gamma_E}}{\log z}
 \left(1+O_f\left(\frac1{\log z}\right)\right),
\]
as asserted in \eqref{eq:V}. In particular,
\[
 \mathfrak S(q_z)(L_z+1)
 =
 V_f(z)(\log z+C_f+1)
 \ll_f 1.
\]
Substitution into Ford's estimate proves
\[
 M_z(v)=V_f(z)\log v+O_f(1),
\]
with an implied constant independent of both $z$ and $v$.

No restriction $v>z$ is needed. Indeed, if $2\le v\le z$,
then every integer $2\le r\le v$ has a prime divisor at most
$z$, and hence $h_z(r)=0$. Thus $M_z(v)=h_z(1)=1$, while
\[
 V_f(z)\log v\ll_f\frac{\log v}{\log z}\ll_f1,
\]
so the estimate holds in this range as well.
\end{proof}

\section{Covering a small set of positions}

\begin{lemma}[A short translate covering the remaining positions]
\label{lem:translate}
There is a constant $\eta_f>0$ such that the following holds for all
sufficiently large $x$. Let
\[
 1\le H\le x(\log x)^2,\qquad
 S\subseteq[1,H]\cap\cZ,\qquad 1\le |S|\le\eta_f x,
 \qquad Q=Q(x).
\]
For every integer $b$, there is an integer $1\le t\le e^x$ such that
each number $f(b+Qt+s)$, $s\in S$, is divisible by a prime $p$ with
\begin{equation}\label{eq:auxiliary-range}
 x^{\alpha_f}<p\le y,\qquad
 y=\exp\left(\frac{x}{(\log x)^5}\right),\qquad
 \alpha_f=m^2+6.
\end{equation}
The threshold for $x$ is independent of $H$, $S$, and $b$.
\end{lemma}

\begin{remark}
There is no restriction $0\le b<Q$. This uniformity will allow us to
place the final interval in $(X/2,X]$.
The lemma asserts divisibility only. Positivity of the polynomial values
and the fact that the divisors are proper will be checked when proving
Theorem~\ref{thm:main}.
\end{remark}

The proof occupies the rest of this section. We construct coefficients
with an exact cancellation identity and use them to define a
nonnegative weight on $1\le t\le e^x$. We then show that the total
weight is positive and exceeds the weighted sum of uncovered positions.
Since the number of uncovered positions at each argument is an integer,
some argument must cover every position.

For the rest of this section, set
\begin{equation}\label{eq:parameters}
 \begin{gathered}
 T=\lfloor e^x\rfloor,\qquad
 z=x^{\alpha_f},\qquad D=e^{x/4},\\
 \beta=\frac1{\log y},\qquad
 u=\beta \log D =\frac14(\log x)^5.
 \end{gathered}
\end{equation}
For the parameters in \eqref{eq:parameters} and sufficiently large $x$,
we have $x<z$, $H<z$, and $z<y$.
Indeed, $\alpha_f=m^2+6\ge7$, and
\[
 \frac{H}{z}\le\frac{(\log x)^2}{x^{\alpha_f-1}}\longrightarrow0,
 \qquad
 \frac{\log y}{\log z}
 =\frac{x}{\alpha_f(\log x)^6}\longrightarrow\infty.
\]
In particular, every prime greater than $z$ is greater than $p_0$
once $x$ is sufficiently large.

Define
\begin{equation}\label{DEF_P}
 P=\prod_{\substack{z<p\le y\\\rho(p)>0}}p,
 \qquad h(r)=h_z(r)\quad(r\mid P),\qquad
 B=\sum_{\substack{r\mid P\\r>1}}\frac{h(r)}{\log r}.
\end{equation}
Thus $P$ is squarefree, $h(1)=1$, and
\[
 h(r)=\prod_{p\mid r}\frac{\rho(p)}{p-\rho(p)}>0
 \qquad(r\mid P).
\]
By \eqref{eq:weighted-mertens},
\[
 \sum_{z<p\le y}\frac{\rho(p)\log p}{p}
 =\log y-\log z+O_f(1)>0
\]
for large $x$. Thus some prime in this range satisfies $\rho(p)>0$,
so $P>1$. Its contribution to $B$ gives $B>0$.
For $r\mid P$, put
\begin{equation}\label{eq:coefficients}
 \begin{gathered}
 a(1)=1,\qquad a(r)=-\frac1{B\log r}\quad(r>1),\\
 A_\theta=\sum_{r\mid P}a(r)^2h(r)r^\theta,
 \qquad A=A_0.
 \end{gathered}
\end{equation}
The sums are finite, although the product $P$ can be very large.

The choice of $a(r)$ can also be obtained by minimization.
For real coefficients $c(r)$, $r\mid P$, $r>1$, subject to
$\sum_{r>1}c(r)h(r)=-1$, Cauchy's inequality gives
\[
 \begin{aligned}
 1
 &=\left(\sum_{\substack{r\mid P\\r>1}}
      c(r)\sqrt{h(r)\log r}\sqrt{\frac{h(r)}{\log r}}\right)^2\\
 &\le
 \left(\sum_{\substack{r\mid P\\r>1}}c(r)^2h(r)\log r\right)B.
 \end{aligned}
\]
Equality holds when $c(r)\log r$ is constant. The cancellation
constraint then forces this constant to be $-1/B$, giving exactly
\eqref{eq:coefficients}. These coefficients therefore minimize the
logarithmically weighted square sum subject to that constraint.

\subsection{Estimates for the coefficients}

\begin{lemma}\label{lem:coefficients}
As $x\to\infty$,
\begin{equation}\label{eq:B-limit}
 B=V_f(z)\log\frac{\log y}{\log z}
       +O_f((\log z)^{-1})
  =\frac{e^{-\gamma_E}}{\alpha_f}+o_f(1).
\end{equation}
For all sufficiently large $x$,
\begin{equation}\label{eq:coefficients-basic}
 1\le A=1+O_f((\log x)^{-2}),\qquad
 |a(r)|\le1,\qquad \sum_{r\mid P}a(r)h(r)=0.
\end{equation}
There is a fixed constant $C\ge1$, depending only on $f$, such that
for $0\le\theta\le2\beta$,
\begin{equation}\label{eq:coefficients-moments}
 \sum_{\substack{r\mid P\\r>1}}|a(r)|h(r)r^\theta\le C,
 \qquad A_\theta\le A+C\theta.
\end{equation}
Finally, uniformly for $p\mid P$,
\begin{equation}\label{eq:coefficients-prime}
 \sum_{\substack{r\mid P\\p\mid r}}|a(r)|h(r)\ll_f\frac1p.
\end{equation}
\end{lemma}

\begin{proof}
\smallskip
\noindent\emph{Step 1: a bound for the Euler products.}
Let
\[
 U_\theta=\sum_{r\mid P}h(r)r^\theta.
\]
Divisors of the squarefree integer $P$ correspond to subsets of
its prime factors. Multiplicativity therefore gives
\[
 U_\theta=\prod_{p\mid P}(1+h(p)p^\theta).
\]
Since
\[
 1+h(p)=1+\frac{\rho(p)}{p-\rho(p)}
       =\left(1-\frac{\rho(p)}p\right)^{-1},
\]
and a prime with $\rho(p)=0$ contributes a factor $1$, we obtain
\begin{equation}\label{eq:U-zero}
 U_0=\prod_{z<p\le y}\left(1-\frac{\rho(p)}p\right)^{-1}
     =\frac{\sigma_f(z)}{\sigma_f(y)}
     \asymp_f\frac{\log y}{\log z}.
\end{equation}
The last estimate follows from \eqref{eq:polynomial-mertens}.
For the perturbed product, observe that
\[
 \frac{1+h(p)p^\theta}{1+h(p)}
 =1+\frac{h(p)}{1+h(p)}(p^\theta-1)
 =1+\frac{\rho(p)}p(p^\theta-1).
\]
If $0\le\theta\le2\beta$ and $p\le y$, then
$0\le\theta\log p\le2$. Using
$e^v-1=\int_0^v e^s\,ds$, we have
\[
 0\le p^\theta-1\le e^2\theta\log p.
\]
The inequality $1+t\le e^t$ for $t\ge0$ therefore yields
\begin{equation}\label{eq:U-theta}
 \begin{aligned}
 \frac{U_\theta}{U_0}
 &\le\exp\left(\sum_{p\mid P}
                    \frac{\rho(p)}p(p^\theta-1)\right)\\
 &\le\exp\left(e^2\theta
                    \sum_{z<p\le y}\frac{\rho(p)\log p}{p}\right)
 \ll_f1.
 \end{aligned}
\end{equation}
Indeed, the prime sum is at most $\log y+O_f(1)$ by
\eqref{eq:weighted-mertens}, and $\theta\log y\le2$.
The bound is thus uniform throughout the stated range of $\theta$.

\smallskip
\noindent\emph{Step 2: evaluation of $B$.}
We first remove the divisors exceeding $y$. For such divisors,
$1/\log r\le1/\log y$, so nonnegativity and \eqref{eq:U-zero} give
\begin{equation}\label{eq:B-tail}
 0\le\sum_{\substack{r\mid P\\r>y}}\frac{h(r)}{\log r}
 \le\frac1{\log y}\sum_{r\mid P}h(r)
 =\frac{U_0}{\log y}\ll_f\frac1{\log z}.
\end{equation}
Put
\[
 K_z(v)=M_z(v)-1=\sum_{1<r\le v}h_z(r).
\]
The subtraction removes the term $r=1$. Every $r>1$ with
$h_z(r)\ne0$ is squarefree and has all its prime factors greater
than $z$. Hence $K_z(v)=0$ for $v\le z$. Moreover,
\eqref{eq:summation} gives, uniformly for $v\ge2$,
\[
 K_z(v)=V_f(z)\log v+O_f(1).
\]
If also $r\le y$, each prime factor of $r$ lies in $(z,y]$
and has $\rho(p)>0$. Thus $r\mid P$ and $h(r)=h_z(r)$.
It follows that partial summation with the function $1/\log v$
gives
\[
 \sum_{\substack{r\mid P\\1<r\le y}}\frac{h(r)}{\log r}
 =\sum_{1<r\le y}\frac{h_{z}(r)}{\log r}=\frac{K_z(y)}{\log y}
   +\int_z^y\frac{K_z(v)}{v(\log v)^2}\,dv.
\]
There is no lower boundary term, since $K_z(z)=0$.
Substituting the estimate for $K_z$ and evaluating the integrals,
we obtain
\[
 \begin{aligned}
 \sum_{\substack{r\mid P\\1<r\le y}}\frac{h(r)}{\log r}
 &=V_f(z)+V_f(z)\int_z^y\frac{dv}{v\log v}
   +O_f\left(\frac1{\log y}
          +\int_z^y\frac{dv}{v(\log v)^2}\right)\\
 &=V_f(z)+V_f(z)\log\frac{\log y}{\log z}
   +O_f\left(\frac1{\log z}\right).
 \end{aligned}
\]
Here
\[
 \int_z^y\frac{dv}{v(\log v)^2}
 =\frac1{\log z}-\frac1{\log y}.
\]
Since $V_f(z)\ll_f1/\log z$, the first term is absorbed into
the error. Adding \eqref{eq:B-tail} proves the first equality
in \eqref{eq:B-limit}.

To compute the limit, note that
\[
 \log z=\alpha_f\log x,\qquad
 \log\frac{\log y}{\log z}
 =\log\frac{x}{\alpha_f(\log x)^6}
 =\log x-6\log\log x-\log\alpha_f.
\]
Formula \eqref{eq:V} now gives
\[
 \begin{aligned}
 B
 &=\frac{e^{-\gamma_E}}{\alpha_f\log x}
   \left(1+O_f\left(\frac1{\log x}\right)\right)
   \bigl(\log x-6\log\log x-\log\alpha_f\bigr)
   +O_f\left(\frac1{\log x}\right)\\
 &=\frac{e^{-\gamma_E}}{\alpha_f}
   +O_f\left(\frac{1+\log\log x}{\log x}\right).
 \end{aligned}
\]
In particular, there are constants $b_-,b_+>0$, depending only
on $f$, such that $b_-\le B\le b_+$ for all sufficiently large $x$.

\smallskip
\noindent\emph{Step 3: the unperturbed square sum and cancellation.}
Separating the term $r=1$ in $A$ gives
\[
 A-1=\frac1{B^2}\sum_{\substack{r\mid P\\r>1}}
                         \frac{h(r)}{(\log r)^2}\ge0.
\]
Enlarge the sum to all $r>1$, with weight $h_z(r)$.
For $R>z$, partial summation gives
\[
 \sum_{1<r\le R}\frac{h_z(r)}{(\log r)^2}
 =\frac{K_z(R)}{(\log R)^2}
   +2\int_z^R\frac{K_z(v)}{v(\log v)^3}\,dv.
\]
The boundary term tends to zero as $R\to\infty$, because
$K_z(R)\ll_f1+V_f(z)\log R$. The same bound shows that the
integral converges. Thus
\[
 \begin{aligned}
 0\le A-1
 &\le\frac2{B^2}\int_z^\infty
                     \frac{K_z(v)}{v(\log v)^3}\,dv\\
 &\ll_f V_f(z)\int_z^\infty\frac{dv}{v(\log v)^2}
       +\int_z^\infty\frac{dv}{v(\log v)^3}\\
 &=\frac{O_f(V_f(z))}{\log z}
       +O_f\left(\frac1{(\log z)^2}\right)
 \ll_f\frac1{(\log x)^2}.
 \end{aligned}
\]
We used $B\ge b_-$, $V_f(z)\ll_f1/\log z$, and
$\log z=\alpha_f\log x$.

For a divisor $r>1$ of $P$, at least one prime greater than $z$
divides $r$, so $r>z$. Therefore
\[
 |a(r)|=\frac1{B\log r}
 \le\frac1{b_-\log z}\le1
\]
for large $x$, while $|a(1)|=1$. Also $a(r)<0$ whenever
$r\mid P$ and $r>1$. Finally, the normalization
by $B$ gives the exact identities
\begin{equation}\label{eq:exact-cancellation}
 \sum_{\substack{r\mid P\\r>1}}|a(r)|h(r)
 =\frac1B\sum_{\substack{r\mid P\\r>1}}\frac{h(r)}{\log r}=1,
 \qquad
 \sum_{r\mid P}a(r)h(r)=1-1=0.
\end{equation}
This proves \eqref{eq:coefficients-basic}.

\smallskip
\noindent\emph{Step 4: the perturbed moments.}
Fix $0\le\theta\le2\beta$ and split the sum involving $|a(r)|$
at $y$. If $r\le y$, then $r^\theta\le e^2$, so
\[
 \sum_{\substack{r\mid P\\1<r\le y}}|a(r)|h(r)r^\theta
 \le e^2\sum_{\substack{r\mid P\\r>1}}|a(r)|h(r)=e^2.
\]
If $r>y$, use $|a(r)|\le1/(B\log y)$ instead. Then
\[
 \begin{aligned}
 \sum_{\substack{r\mid P\\r>y}}|a(r)|h(r)r^\theta
 &\le\frac1{B\log y}\sum_{r\mid P}h(r)r^\theta
 =\frac{U_\theta}{B\log y}
 \ll_f\frac1{\log z},
 \end{aligned}
\]
by \eqref{eq:U-zero} and \eqref{eq:U-theta}. Adding the two bounds
proves that the full sum is $O_f(1)$, uniformly in $\theta$.

For $r\ge1$, integration with respect to the exponent gives
\[
 0\le r^\theta-1=\int_0^\theta r^s\log r\,ds
 \le\theta r^\theta\log r.
\]
Also, $a(r)^2\log r=|a(r)|/B$ for $r>1$. Consequently,
\[
 \begin{aligned}
 0\le A_\theta-A
 &=\sum_{\substack{r\mid P\\r>1}}a(r)^2h(r)(r^\theta-1)\\
 &\le\theta\sum_{\substack{r\mid P\\r>1}}
                       a(r)^2h(r)r^\theta\log r\\
 &=\frac\theta B\sum_{\substack{r\mid P\\r>1}}
                       |a(r)|h(r)r^\theta
 \ll_f\theta.
 \end{aligned}
\]
Choose $C\ge1$, depending only on $f$, so that both estimates
in \eqref{eq:coefficients-moments} hold.

\smallskip
\noindent\emph{Step 5: coefficients containing a specified prime.}
Fix $p\mid P$. Since $P$ is squarefree, every divisor $r\mid P$
with $p\mid r$ can be written uniquely as $r=pe$, where
$e\mid P/p$ and $(p,e)=1$. Multiplicativity gives $h(pe)=h(p)h(e)$.
Separating $e=1$, we obtain
\[
 \begin{aligned}
 \sum_{\substack{r\mid P\\p\mid r}}|a(r)|h(r)
 &=\frac{h(p)}B\sum_{e\mid P/p}\frac{h(e)}{\log(pe)}\\
 &=\frac{h(p)}B\left(\frac1{\log p}
       +\sum_{\substack{e\mid P/p\\e>1}}
                              \frac{h(e)}{\log(pe)}\right)\\
 &\le\frac{h(p)}B\left(\frac1{\log p}
       +\sum_{\substack{e\mid P\\e>1}}\frac{h(e)}{\log e}\right)
 =\frac{h(p)}B\left(\frac1{\log p}+B\right).
 \end{aligned}
\]
The inequality uses $\log(pe)\ge\log e$ for $e>1$ and enlarges
the set of divisors. Since $p>z>2m$ and $0<\rho(p)\le m$,
\[
 0<h(p)=\frac{\rho(p)}{p-\rho(p)}\le\frac{2m}{p}.
\]
Together with $B\ge b_-$ and $\log p\ge\log z$, this proves
\eqref{eq:coefficients-prime} uniformly for all $p\mid P$.
\end{proof}

Fix $C$ as in Lemma~\ref{lem:coefficients}, and now choose
\begin{equation}\label{eq:eta-choice}
 \eta_f=\min\left\{\frac14,\frac1{16C}\right\}.
\end{equation}
The coefficient estimates do not involve $H$, $S$, or $b$, so
the choice depends only on $f$. Write $k=|S|$ and enumerate
$S=\{s_1,\ldots,s_k\}$, with distinct elements. Since
$1\le k\le\eta_f x\le x$, multiplication by $C\beta$ and the bound
$\eta_f\le1/(16C)$ in \eqref{eq:eta-choice} give
\begin{equation}\label{eq:budget-margin}
 Ck\beta\le C\eta_f x\beta
 \le\frac{x}{16\log y}=\frac{1}{16} (\log x)^{5}=\frac u4.
\end{equation}
This leaves room for the increase in the moments when the exponent
changes from $0$ to $\beta$ or $2\beta$.
\subsection{Local factors and the truncated weight}

Put 
\begin{equation}\label{DEF_F_i}
F_i(t)=g(b+Qt+s_i),\qquad i=1,\ldots, k.
\end{equation} For $p\mid P$, define
\begin{equation}\label{psi.def}
 \psi_{i,p}(t)=\frac{\rho(p)-p\ind_{p\mid F_i(t)}}{p-\rho(p)},
 \qquad
 \psi_{i,r}(t)=\prod_{p\mid r}\psi_{i,p}(t),\qquad
 \psi_{i,1}(t)=1.
\end{equation}
The last identity follows from the empty-product convention.
Primes dividing $P$ exceed $z>x$, while prime divisors of $Q=Q(x)$
are at most $x$. Hence $(p,Q)=1$, and the map
\[
 t\bmod p\longmapsto b+Qt\bmod p
\]
is a bijection: the unique preimage of $a\bmod p$ is
$Q^{-1}(a-b)\bmod p$, where $Q^{-1}$ denotes the inverse
of $Q$ modulo $p$. Consequently $F_i(t)$ has exactly $\rho(p)$
roots modulo $p$: each root $a$ of $g$ gives the unique root
$t\equiv Q^{-1}(a-b-s_i)\pmod p$ of $F_i$.

The root sets of $F_i$ and $F_j$ are disjoint when $i\ne j$.
Indeed, a common root $t$ would give a common root $b+Qt$
of $g(v+s_i)$ and $g(v+s_j)$ modulo $p$. This is impossible
by Lemma~\ref{lem:separation}, since $p>z=x^{\alpha_f}$ and
$s_i,s_j$ are distinct elements of $[1,H]$.

\smallskip
\noindent\emph{The local moments.}
Let $t$ be uniform modulo $p$, and write
$q_p=\rho(p)/p$ and $h_p=h(p)=q_p/(1-q_p)$. The definition gives
\[
 \psi_{i,p}(t)=
 \begin{cases}
 -1,&p\mid F_i(t),\\[1mm]
 h_p,&p\nmid F_i(t).
 \end{cases}
\]
These two alternatives have probabilities $q_p$ and $1-q_p$.
Their first and second moments are therefore
\[
 \begin{aligned}
 \E_p\psi_{i,p}
 &=-q_p+(1-q_p)\frac{q_p}{1-q_p}=0,\\
 \E_p\psi_{i,p}^2
 &=q_p+(1-q_p)\frac{q_p^2}{(1-q_p)^2}
 =\frac{q_p}{1-q_p}=h_p.
 \end{aligned}
\]
If $i\ne j$, the two root events are disjoint. On the roots
of $F_i$ the product $\psi_{i,p}\psi_{j,p}$ equals $-h_p$,
and the same is true on the roots of $F_j$. The complement of
these two sets has probability $1-2q_p$, and on it the product
equals $h_p^2$. Hence
\[
 \begin{aligned}
 \E_p(\psi_{i,p}\psi_{j,p})
 &=-2q_ph_p+(1-2q_p)h_p^2\\
 &=\frac{-2q_p^2(1-q_p)+(1-2q_p)q_p^2}{(1-q_p)^2}
 =-\frac{q_p^2}{(1-q_p)^2}=-h_p^2.
 \end{aligned}
\]
Thus
\begin{equation}\label{eq:local-moments}
 \E_p\psi_{i,p}=0,\qquad
 \E_p\psi_{i,p}^2=h(p),\qquad
 \E_p(\psi_{i,p}\psi_{j,p})=-h(p)^2\quad(i\ne j).
\end{equation}
In particular, the local factors belonging to different positions
at the same prime are not independent. The independence used below
concerns distinct prime moduli.

\smallskip
\noindent\emph{Cancellation at an uncovered position.}
We call position $i$ \emph{uncovered at $t$} if $(F_i(t),P)=1$.
This is equivalent to $p\nmid F_i(t)$ for every $p\mid P$.
For such a position, each local factor equals $h(p)$, and hence
\[
 \psi_{i,r}(t)=\prod_{p\mid r}h(p)=h(r)
 \qquad(r\mid P).
\]
The equality also holds for $r=1$. Thus \eqref{eq:exact-cancellation}
implies
\[
 \sum_{r\mid P}a(r)\psi_{i,r}(t)
 =\sum_{r\mid P}a(r)h(r)=0.
\]
The full product
\[
 \prod_{i=1}^k\left(\sum_{r\mid P}a(r)\psi_{i,r}(t)\right)
\]
therefore vanishes whenever at least one position is uncovered.
The unrestricted expansion, however, includes divisors of $P$
far exceeding the averaging length $T$. To pass from complete averages
to averages over $1\le t\le T$, we retain only tuples whose coordinate
product is bounded and whose coordinates are pairwise coprime.

For $1\le j\le k$, set
\[
 \cR_j=\left\{(r_1,\ldots,r_j):
 \begin{array}{l}
 r_i\mid P,\quad(r_i,r_\ell)=1\ (i\ne\ell),\\
 r_1\cdots r_j\le D
 \end{array}\right\}.
\]
Let $\cR_0$ consist of the empty tuple, whose product is $1$.
Pairwise coprimality means that a prime can occur in at most one
coordinate of a tuple. The product bound ensures that, when two
terms are multiplied, the resulting modulus is at most $D^2$.
Define
\begin{equation}\label{eq:weight}
 R(t)=\sum_{\bm r\in\cR_k}
                \prod_{i=1}^k a(r_i)\psi_{i,r_i}(t),
 \qquad w(t)=R(t)^2\ge0,
\end{equation}
and introduce the finite averages
\begin{equation}\label{eq:S1S2}
 \begin{aligned}
 S_1&=\frac1T\sum_{t=1}^T w(t),\\
 S_2&=\frac1T\sum_{t=1}^T w(t)
                  \sum_{i=1}^k\ind_{(F_i(t),P)=1}.
 \end{aligned}
\end{equation}
The inner sum in $S_2$, defined in \eqref{eq:S1S2}, is the
nonnegative integer counting uncovered positions. If it were at
least $1$ for every $1\le t\le T$, multiplying by $w(t)\ge0$
and summing would give $S_2\ge S_1$. It is therefore enough to prove
$S_2<S_1$. We shall show, uniformly in the parameters of the lemma,
that $S_1$ is comparable to $A^k$ and $S_2=o_f(A^k)$.

\subsection{From complete averages to finite averages}\label{subsec:finite-averages}

For a $P$-periodic function $U$, write
\begin{equation}\label{DEF_E}
 \E U=\frac1P\sum_{t=0}^{P-1}U(t).
\end{equation}
We shall compare this complete average with
$T^{-1}\sum_{t=1}^T U(t)$ for the functions arising from the weight.

\smallskip
\noindent\emph{Uniformity and independence modulo the auxiliary primes.}
If $p\mid P$ and $a\bmod p$ is fixed, exactly $P/p$ integers
in $\{0,\ldots,P-1\}$ are congruent to $a$ modulo $p$.
Indeed, these integers are
$a,a+p,\ldots,a+(P/p-1)p$ when $0\le a<p$.
Thus each residue modulo $p$ has weight $1/p$ under $\E$,
and for every $p$-periodic function $U_p$,
\[
 \E U_p=\frac1p\sum_{a=0}^{p-1}U_p(a)=\E_pU_p.
\]
More generally, let $p_1,\ldots,p_\ell$ be distinct prime
divisors of $P$, and prescribe one residue $a_j\bmod p_j$
for each of them. The Chinese remainder theorem gives one class
modulo $q=p_1\cdots p_\ell$. As $q\mid P$, this class has
exactly $P/q$ representatives in a complete period modulo $P$.
Consequently
\[
 \E\prod_{j=1}^\ell\ind_{t\equiv a_j\pmod{p_j}}
 =\frac1q=\prod_{j=1}^\ell\frac1{p_j}.
\]
Thus residues at distinct primes are independent under the complete
average.

In particular, for functions $U_p$ depending only on $t\bmod p$,
the Chinese remainder bijection gives
\[
 \begin{aligned}
 \E\prod_{p\mid P}U_p(t)
 &=\frac1P\sum_{\substack{0\le a_p<p\\\text{for each }p\mid P}}
                     \prod_{p\mid P}U_p(a_p)\\
 &=\prod_{p\mid P}\left(\frac1p\sum_{a=0}^{p-1}U_p(a)\right)
 =\prod_{p\mid P}\E_pU_p.
 \end{aligned}
\]
In applying this identity, factors at the same prime must be grouped
together. Independence holds across prime moduli; factors at different
positions may be dependent.

\smallskip
\noindent\emph{Expansion into divisibility indicators.}
The identity
\begin{equation}\label{eq:local-expansion}
 \psi_{i,p}(t)=h(p)-(1+h(p))\ind_{p\mid F_i(t)}
\end{equation}
follows because $1+h(p)=p/(p-\rho(p))$.
Using \eqref{psi.def} and expanding the product of the identities
\eqref{eq:local-expansion}, we obtain, for every divisor $r\mid P$,
\begin{equation}\label{expansion.psi}
 \psi_{i,r}(t)
 =\sum_{d\mid r}(-1)^{\omega(d)}h(r/d)
                \prod_{p\mid d}(1+h(p))\,
                \ind_{d\mid F_i(t)}.
\end{equation}
Here
\[
 \prod_{p\mid d}\ind_{p\mid F_i(t)}
 =\ind_{d\mid F_i(t)}.
\]
The product is $1$ precisely when every prime divisor of $d$
divides $F_i(t)$. Squarefreeness of $d$ makes this equivalent to
$d\mid F_i(t)$; otherwise both sides vanish.

The formula \eqref{expansion.psi} also holds for $r=1$. In that case the only
divisor is $d=1$, and
\[
 (-1)^{\omega(1)}h(1)
 \prod_{p\mid1}(1+h(p))\,
 \ind_{1\mid F_i(t)}
 =1,
\]
in agreement with the definition $\psi_{i,1}(t)=1$.

For $\bm r\in\cR_k$, selecting a divisor $d_i\mid r_i$ in each
coordinate gives pairwise coprime $d_i$ with
$d_1\cdots d_k\le r_1\cdots r_k\le D$. Collecting coefficients
therefore gives an expansion (see \eqref{eq:weight})
\begin{equation}\label{eq:lambda-expansion}
 R(t)=\sum_{\bm d\in\cR_k}\lambda_{\bm d}
                          \prod_{i=1}^k\ind_{d_i\mid F_i(t)}.
\end{equation}
The coefficients can be written explicitly as
\[
 \lambda_{\bm d}
 =\sum_{\substack{\bm r\in\cR_k\\d_i\mid r_i\ (1\le i\le k)}}
   \prod_{i=1}^k\left(
    a(r_i)(-1)^{\omega(d_i)}h(r_i/d_i)
                  \prod_{p\mid d_i}(1+h(p))\right).
\]

To bound the sum of absolute coefficients, apply the triangle
inequality before collecting equal indicator products.

For $d\mid r\mid P$, write
\[
 c(r,d)=(-1)^{\omega(d)}h(r/d)
                         \prod_{p\mid d}(1+h(p)).
\]
Thus
\[
 \psi_{i,r}(t)
 =\sum_{d\mid r}c(r,d)\ind_{d\mid F_i(t)}.
\]
Since $h(e)>0$ for every $e\mid P$, we have
\[
 |c(r,d)|=h(r/d)\prod_{p\mid d}(1+h(p)).
\]
The formula for $\lambda_{\bm d}$ becomes
\[
 \lambda_{\bm d}
 =
 \sum_{\substack{\bm r\in\cR_k\\
                  d_i\mid r_i\ (1\le i\le k)}}
       \prod_{i=1}^k a(r_i)c(r_i,d_i).
\]
Contributions from different $\bm r$ to a fixed $\lambda_{\bm d}$
give, by the triangle inequality,
\[
 |\lambda_{\bm d}|
 \le
 \sum_{\substack{\bm r\in\cR_k\\
                  d_i\mid r_i\ (1\le i\le k)}}
       \prod_{i=1}^k |a(r_i)|\,|c(r_i,d_i)|.
\]

Sum over $\bm d\in\cR_k$ and interchange the finite sums.
For $\bm r\in\cR_k$, choosing $d_i\mid r_i$ for every $i$
always produces a tuple $\bm d$ in $\cR_k$, since
\[
 (d_i,d_j)=1\quad(i\ne j),
 \qquad
 d_1\cdots d_k\le r_1\cdots r_k\le D,
\]
and each $d_i$ divides $P$. For a fixed $\bm r$, the choices
$d_i\mid r_i$ are therefore independent, giving
\[
 \begin{aligned}
 \sum_{\bm d\in\cR_k}|\lambda_{\bm d}|
 &\le
 \sum_{\bm r\in\cR_k}
 \sum_{\substack{d_1\mid r_1,\ldots,d_k\mid r_k}}
       \prod_{i=1}^k |a(r_i)|\,|c(r_i,d_i)|\\
 &=
 \sum_{\bm r\in\cR_k}
       \prod_{i=1}^k
       \left(
       |a(r_i)|\sum_{d_i\mid r_i}|c(r_i,d_i)|
       \right).
 \end{aligned}
\]

For each squarefree divisor $r\mid P$, the inner sum factors as
\[
 \begin{aligned}
 \sum_{d\mid r}|c(r,d)|
 &=\sum_{d\mid r}h(r/d)\prod_{p\mid d}(1+h(p))\\
 &=\prod_{p\mid r}\bigl(h(p)+(1+h(p))\bigr)\\
 &=\prod_{p\mid r}(1+2h(p)).
 \end{aligned}
\]
These equalities also hold for $r=1$, since all sides are $1$.

Substituting this identity gives
\[
 \sum_{\bm d\in\cR_k}|\lambda_{\bm d}|
 \le
 \sum_{\bm r\in\cR_k}
       \prod_{i=1}^k
       \left(
       |a(r_i)|\prod_{p\mid r_i}(1+2h(p))
       \right).
\]
By \eqref{eq:coefficients-basic}, $|a(r_i)|\le1$.
Moreover, the coordinates of $\bm r$ are pairwise
coprime, so each prime divisor of $r_1\cdots r_k$
occurs in exactly one coordinate. Consequently,
\[
 \sum_{\bm d\in\cR_k}|\lambda_{\bm d}|
 \le
 \sum_{\bm r\in\cR_k}
       \prod_{p\mid r_1\cdots r_k}(1+2h(p)).
\]

To bound the product, put $n=r_1\cdots r_k$.
We have $n\mid P$, $n\le D$, and every prime divisor of $n$
exceeds $z$, so
\[
 \omega(n)\log z
 \le\sum_{p\mid n}\log p
 =\log n\le\log D.
\]
Since $h(p)\le2m/p$, the inequality $1+s\le e^s$ for
$s\ge0$ yields
\[
 \begin{aligned}
 \prod_{p\mid n}(1+2h(p))
 &\le\exp\left(2\sum_{p\mid n}h(p)\right)\\
 &\le\exp\left(4m\sum_{p\mid n}\frac1p\right)\\
 &\le\exp\left(\frac{4m\omega(n)}z\right)\\
 &\le\exp\left(\frac{4m\log D}{z\log z}\right).
 \end{aligned}
\]
For our parameters,
\[
 \frac{4m\log D}{z\log z}
 =
 \frac{m}{\alpha_f x^{\alpha_f-1}\log x}
 \longrightarrow0,
\]
because $\log D=x/4$, $z=x^{\alpha_f}$, and
$\alpha_f=m^2+6\ge7$. Thus
\[
 \prod_{p\mid n}(1+2h(p))\le2
\]
for all sufficiently large $x$, uniformly over
$n\mid P$ with $n\le D$. The estimate includes $n=1$,
for which the product is empty and equals $1$.

We now count the tuples with product $n$. Fix $n\mid P$ with
$n\le D$. If $n>1$, squarefreeness of $P$ allows us to write
\[
 n=p_1\cdots p_\ell,\qquad \ell=\omega(n),
\]
where $p_1,\ldots,p_\ell$ are distinct primes.

A tuple $(r_1,\ldots,r_k)\in\cR_k$ with $r_1\cdots r_k=n$
assigns each prime $p_j$ to a unique coordinate: the prime divides
the product $n$, and coprimality prevents it from dividing two
coordinates. This defines a map
\[
 \varphi:\{1,\ldots,\ell\}\longrightarrow\{1,\ldots,k\},
\]
where $\varphi(j)=i$ means that $p_j\mid r_i$.

Conversely, every map
$\varphi:\{1,\ldots,\ell\}\to\{1,\ldots,k\}$ determines a tuple by
\[
 r_i=\prod_{\substack{1\le j\le\ell\\\varphi(j)=i}}p_j
 \qquad(1\le i\le k),
\]
with an empty product interpreted as $1$. Each coordinate
is a divisor of $n$, and hence a divisor of $P$. A prime
$p_j$ occurs only in the coordinate indexed by $\varphi(j)$,
so the coordinates are pairwise coprime. Moreover, every
$p_j$ occurs in one of them, and therefore
\[
 r_1\cdots r_k=\prod_{j=1}^{\ell}p_j=n\le D.
\]
Thus the resulting tuple belongs to $\cR_k$.

The constructions are inverse: the map records the prime factors of
each coordinate, and the product formula puts each $p_j$ back into
coordinate $\varphi(j)$. Thus tuples in $\cR_k$ with product $n$
are in bijection with the maps
\[
 \varphi:\{1,\ldots,\ell\}\longrightarrow\{1,\ldots,k\}.
\]

For each of the $\ell$ arguments $j$, the value $\varphi(j)$
is an independent choice from $\{1,\ldots,k\}$. There are
therefore $k^\ell$ maps, and
\[
 \#\{\bm r\in\cR_k:r_1\cdots r_k=n\}
 =k^\ell=k^{\omega(n)}.
\]
For $n=1$, the only possible tuple is $(1,\ldots,1)$,
because all coordinates are positive integers. Since
$\omega(1)=0$ and $k\ge1$, the same formula gives
$k^{\omega(1)}=1$. 

Grouping the preceding estimate by $n$ therefore gives
\[
 \sum_{\bm d\in\cR_k}|\lambda_{\bm d}|
 \le2\sum_{\substack{n\mid P\\n\le D}}k^{\omega(n)}.
\]
Since $k\ge1$ and $\omega(n)\le\log D/\log z$,
\[
 k^{\omega(n)}
 =\exp(\omega(n)\log k)
 \le\exp\left(\frac{\log D\log k}{\log z}\right)
 =D^{\log k/\log z}.
\]
We conclude that
\begin{equation}\label{eq:l1}
 \begin{aligned}
 \sum_{\bm d\in\cR_k}|\lambda_{\bm d}|
 &\le2\sum_{\substack{n\mid P\\n\le D}}k^{\omega(n)}\\
 &\le2D^{1+\log k/\log z}
 \le D^{3/2}.
 \end{aligned}
\end{equation}
For the last inequality, recall that $1\le k\le x$ and
$\log z=\alpha_f\log x$, so
\[
 \frac{\log k}{\log z}
 \le\frac1{\alpha_f}\le\frac17.
\]
Therefore
\[
 2D^{1+\log k/\log z}\le2D^{8/7}\le D^{3/2}
\]
as soon as $D^{5/14}\ge2$. Since $D=e^{x/4}\to\infty$,
this holds for all sufficiently large $x$, uniformly
in $k$.

\smallskip
\noindent\emph{Counting the residue classes in a product of indicators.}
With $F_i$ and $P$ as in \eqref{DEF_F_i} and \eqref{DEF_P},
put, for $\bm d\in\cR_k$,
\[
 J_{\bm d}(t)=\prod_{i=1}^k\ind_{d_i\mid F_i(t)}.
\]
Expanding the square in \eqref{eq:lambda-expansion} gives
\begin{equation}\label{R^2}
 R(t)^2=\sum_{\bm d,\bm e\in\cR_k}
        \lambda_{\bm d}\lambda_{\bm e}J_{\bm d}(t)J_{\bm e}(t).
\end{equation}
We compare the two averages term by term for
$J_{\bm d}J_{\bm e}$ and sum the errors against the absolute
values of the coefficients.

Fix two tuples $\bm d,\bm e\in\cR_k$. By definition,
\[
 J_{\bm d}(t)J_{\bm e}(t)
 =\prod_{i=1}^k
       \ind_{d_i\mid F_i(t)}\ind_{e_i\mid F_i(t)}.
\]
Thus this product equals $1$ if and only if all the conditions
\[
 d_i\mid F_i(t),\qquad e_i\mid F_i(t)
 \qquad(1\le i\le k)
\]
hold simultaneously. Put
\[
 q=\operatorname{lcm}(d_1,\ldots,d_k,e_1,\ldots,e_k).
\]
Since every coordinate divides the squarefree integer $P$, we have
$q\mid P$, and $q$ is squarefree. The coordinates are pairwise
coprime within each tuple, so
\begin{equation}\label{D^2}
 \begin{aligned}
 q&=\operatorname{lcm}(d_1\cdots d_k,e_1\cdots e_k)\\
  &\le(d_1\cdots d_k)(e_1\cdots e_k)\le D^2.
 \end{aligned}
\end{equation}
There are two cases, according to whether a prime is assigned
to distinct positions.

\smallskip
\noindent\emph{Case 1: $(d_i,e_j)>1$ for some $i\ne j$.}
Choose a prime $p\mid(d_i,e_j)$. Then $p\mid P$, so $p>z$.
If $J_{\bm d}(t)J_{\bm e}(t)=1$ for some integer $t$, we would have
\[
 p\mid F_i(t)=g(b+Qt+s_i),\qquad
 p\mid F_j(t)=g(b+Qt+s_j).
\]
Writing $v=b+Qt\bmod p$, this would make $v$ a common root of
$g(v+s_i)$ and $g(v+s_j)$ modulo $p$. The shifts $s_i$ and $s_j$
are distinct elements of $[1,H]$, and $p>x^{\alpha_f}$.
Lemma~\ref{lem:separation} rules out such a common root.
Therefore
\[
 J_{\bm d}(t)J_{\bm e}(t)=0
 \qquad\text{for every }t\in\cZ.
\]
In this case both its complete average and its finite average
are zero.

\smallskip
\noindent\emph{Case 2: $(d_i,e_j)=1$ for all $i\ne j$.}
For each $i$, define
\[
 q_i=\operatorname{lcm}(d_i,e_i).
\]
The two conditions belonging to the same position are equivalent
to a single condition:
\[
 d_i\mid F_i(t)\ \text{and}\ e_i\mid F_i(t)
 \quad\Longleftrightarrow\quad q_i\mid F_i(t).
\]
Moreover, membership in $\cR_k$ gives coprimality within each tuple,
and the assumption of this case gives coprimality between distinct
coordinates of the two tuples. Thus
\[
 (d_i,d_j)=(e_i,e_j)=(d_i,e_j)=(e_i,d_j)=1
 \qquad(i\ne j).
\]
It follows that the $q_i$ are pairwise coprime. Indeed, a prime
dividing both $q_i$ and $q_j$ would divide one of $d_i,e_i$ and
one of $d_j,e_j$, contradicting one of these four equalities.
Consequently,
\[
 q_1\cdots q_k
 =\operatorname{lcm}(q_1,\ldots,q_k)=q.
\]

We now count the roots of $F_i$ modulo $q_i$.
For every prime $p\mid q_i$, we have $p\mid P$, hence
$p>z>x$ and $(Q,p)=1$. The congruence
\[
 F_i(t)=g(b+Qt+s_i)\equiv0\pmod p
\]
has exactly $\rho(p)$ solutions: each root $a\bmod p$ of $g$
gives the unique solution
\[
 t\equiv Q^{-1}(a-b-s_i)\pmod p,
\]
where $Q^{-1}$ is the inverse of $Q$ modulo $p$. Distinct roots
$a$ give distinct residues $t$, and every solution arises in
this way.

Since $q_i$ is squarefree, the condition $q_i\mid F_i(t)$ is
equivalent to imposing these congruences at every prime dividing
$q_i$. Choosing one of the $\rho(p)$ solutions at each such prime
and applying the Chinese remainder theorem gives exactly
\[
 \prod_{p\mid q_i}\rho(p)
\]
solutions modulo $q_i$. Let $\mathcal A_i$ be their representatives
in $\{0,\ldots,q_i-1\}$. When $q_i=1$, there is no divisibility
restriction; we take $\mathcal A_i=\{0\}$, and its cardinality is
the empty product $1$.

The moduli $q_1,\ldots,q_k$ are pairwise coprime. The Chinese
remainder theorem therefore supplies, for every choice
\[
 (a_1,\ldots,a_k)\in\mathcal A_1\times\cdots\times\mathcal A_k,
\]
one residue class $a\bmod q$ satisfying
$a\equiv a_i\pmod{q_i}$ for all $i$. Different choices produce
different classes: if the classes modulo $q$ agreed, their
reductions modulo each $q_i$ would agree as well. Conversely,
every class satisfying all the divisibility conditions reduces
to an element of each $\mathcal A_i$.

Let $\mathcal A_{\bm d,\bm e}\subseteq\{0,\ldots,q-1\}$ denote
the resulting set of representatives. The preceding bijection gives
\[
 \begin{aligned}
 |\mathcal A_{\bm d,\bm e}|
 &=\prod_{i=1}^k|\mathcal A_i|
 =\prod_{i=1}^k\prod_{p\mid q_i}\rho(p)\\
 &=\prod_{p\mid q}\rho(p)=:\nu(q).
 \end{aligned}
\]
Here we used the pairwise coprimality of the $q_i$. For $q=1$, we have $\nu(1)=1$ and
$\mathcal A_{\bm d,\bm e}=\{0\}$.

Combining the two cases, we obtain
\[
 J_{\bm d}(t)J_{\bm e}(t)=
 \begin{cases}
  0,&\text{if }(d_i,e_j)>1\text{ for some }i\ne j,\\[2mm]
  \displaystyle\sum_{a\in\mathcal A_{\bm d,\bm e}}
                 \ind_{t\equiv a\pmod q},
    &\text{if }(d_i,e_j)=1\text{ for all }i\ne j.
 \end{cases}
\]
In the second case the classes are distinct, so at most one indicator
is $1$. Their sum is the indicator of the solution set; for $q=1$
it is identically $1$. In general, there are $\nu(q)$ residue
classes, and this number need not be one.

For later use, we bound this number uniformly. Since $\rho(p)\le m$
and every prime divisor of $q$ exceeds $z$, from \eqref{D^2} we obtain
\[
 \nu(q)\le m^{\omega(q)},\qquad
 \omega(q)\log z\le\sum_{p\mid q}\log p
 =\log q\le2\log D.
\]
It follows that
\[
 \nu(q)\le
 \exp\left(\frac{2\log D\log m}{\log z}\right)
 =D^{2\log m/\log z}.
\]
When $m=1$, every prime factor of $P$ has $\rho(p)=1$, so
$\nu(q)=1$, in agreement with this bound. The bound also includes
$q=1$.

\smallskip
\noindent\emph{The error in passing to the finite interval.}
Using $\E$ as defined in \eqref{DEF_E}, consider a fixed pair
$\bm d,\bm e$. In Case~1 both averages vanish. In Case~2,
take one of the allowed classes $a\bmod q$. Since $q\mid P$,
there are exactly $P/q$ integers in $\{0,\ldots,P-1\}$ belonging
to that class. Hence
\[
 \E\,\ind_{t\equiv a\pmod q}
 =\frac1P\cdot\frac Pq=\frac1q.
\]
Summing this identity over the $\nu(q)$ distinct allowed classes gives
\[
 \E(J_{\bm d}J_{\bm e})
 =\sum_{a\in\mathcal A_{\bm d,\bm e}}\frac1q
 =\frac{\nu(q)}q.
\]

To count representatives in the shorter interval, write
\[
 T=\ell q+v,\qquad 0\le v<q.
\]
Each of the $\ell$ complete blocks of $q$ consecutive integers
contains exactly one representative of $a\bmod q$, and the final
block of $v$ integers contains at most one. Therefore
\[
 \#\{1\le t\le T:t\equiv a\pmod q\}
 \in\{\ell,\ell+1\}.
\]
Since $T/q=\ell+v/q$, we obtain
\[
 \left|\#\{1\le t\le T:t\equiv a\pmod q\}-\frac Tq\right|\le1.
\]
After division by $T$, this says
\[
 \left|\frac1T\sum_{t=1}^T\ind_{t\equiv a\pmod q}
                  -\frac1q\right|\le\frac1T.
\]
The estimate is independent of the representative $a$.
For $q=1$ the count is exactly $T$, and the error is zero.

In Case~2, summing over the allowed classes and applying the triangle
inequality gives
\[
 \begin{aligned}
 &\left|\frac1T\sum_{t=1}^TJ_{\bm d}(t)J_{\bm e}(t)
                   -\E(J_{\bm d}J_{\bm e})\right|\\
 &\quad=
 \left|\sum_{a\in\mathcal A_{\bm d,\bm e}}
       \left(\frac1T\sum_{t=1}^T\ind_{t\equiv a\pmod q}
                           -\frac1q\right)\right|\\
 &\quad\le\sum_{a\in\mathcal A_{\bm d,\bm e}}\frac1T
 =\frac{\nu(q)}T
 \le\frac{D^{2\log m/\log z}}T.
 \end{aligned}
\]
The same final bound holds in Case~1, since its left-hand side
is zero. We have therefore proved it for every pair
$\bm d,\bm e\in\cR_k$.

By linearity of the two averages and the expansion of $R^2$ (see also \eqref{eq:weight}, \eqref{eq:S1S2}, and \eqref{R^2}),
\[
 \begin{aligned}
 S_1-\E R^2
 =\sum_{\bm d,\bm e\in\cR_k}
 \lambda_{\bm d}\lambda_{\bm e}
 \left(\frac1T\sum_{t=1}^T J_{\bm d}(t)J_{\bm e}(t)
                         -\E(J_{\bm d}J_{\bm e})\right).
 \end{aligned}
\]
Applying the triangle inequality and then \eqref{eq:l1}, we find
\[
 \begin{aligned}
 |S_1-\E R^2|
 &\le\frac{D^{2\log m/\log z}}T
       \sum_{\bm d,\bm e\in\cR_k}
                         |\lambda_{\bm d}|\,|\lambda_{\bm e}|\\
 &=\frac{D^{2\log m/\log z}}T
       \left(\sum_{\bm d\in\cR_k}|\lambda_{\bm d}|\right)^2\\
 &\le\frac{D^{2\log m/\log z}}T(D^{3/2})^2
 =\frac{D^{3+2\log m/\log z}}T.
 \end{aligned}
\]
Accordingly,
\begin{equation}\label{eq:averaging-error}
 \left|S_1-\E R^2\right|\le\mathcal E_x,
 \qquad
 \mathcal E_x:=\frac{D^{3+2\log m/\log z}}T
       =\exp\bigl(-x/4+o_f(x)\bigr).
\end{equation}
To verify the last equality, use $D=e^{x/4}$,
$\log z=\alpha_f\log x$, and
$\log T=x+O(e^{-x})$, which follows from $T=\lfloor e^x\rfloor$.
Then
\[
 \begin{aligned}
 \log\mathcal E_x
 &=\left(3+\frac{2\log m}{\alpha_f\log x}\right)\frac x4
                       -x+O(e^{-x})\\
 &=-\frac x4+\frac{x\log m}{2\alpha_f\log x}+O(e^{-x}).
 \end{aligned}
\]
For fixed $f$, the second term is $o_f(x)$; it vanishes when
$m=1$. In particular, $\mathcal E_x\le e^{-x/8}$ for all
sufficiently large $x$.

The same argument applies if any subset of the terms in
\eqref{eq:lambda-expansion} is retained. More precisely, for
$\mathcal J\subseteq\cR_k$, set
\[
 R_{\mathcal J}(t)=\sum_{\bm d\in\mathcal J}
                                     \lambda_{\bm d}J_{\bm d}(t).
\]
Every retained pair has the same error bound, and
$\sum_{\bm d\in\mathcal J}|\lambda_{\bm d}|
 \le\sum_{\bm d\in\cR_k}|\lambda_{\bm d}|$. Hence
\begin{equation}\label{SUM_R_SUBSET}
 \left|\frac1T\sum_{t=1}^TR_{\mathcal J}(t)^2
                          -\E R_{\mathcal J}^2\right|
 \le\frac{D^{2\log m/\log z}}T
       \left(\sum_{\bm d\in\mathcal J}|\lambda_{\bm d}|\right)^2
 \le\mathcal E_x.
\end{equation}
We shall use this bound for majorants of uncovered positions.
The estimates are uniform in $b$ and the allowed shifts: only the
locations of the residue classes depend on these parameters.
Their number is bounded in terms of $q$ and $m$, and each class
contributes an error of at most $1/T$.

\subsection{A diagonal estimate for the complete average}

\begin{lemma}\label{lem:diagonal}
Let $0\le j\le k$, and let $c_{\bm r}$, $\bm r\in\cR_j$, be
arbitrary real coefficients. Then
\begin{align}
 &\left|\E\left(\sum_{\bm r\in\cR_j}c_{\bm r}
                 \prod_{i=1}^j\psi_{i,r_i}\right)^2
       -\sum_{\bm r\in\cR_j}c_{\bm r}^2\prod_{i=1}^j h(r_i)\right|
 \notag\\
 &\hspace{20mm}\ll_f\frac{k\log D}{z}
              \sum_{\bm r\in\cR_j}c_{\bm r}^2\prod_{i=1}^j h(r_i).
 \label{eq:diagonal}
\end{align}
\end{lemma}

\begin{proof}
If $j=0$, the set $\cR_0$ contains only the empty tuple. The
function being averaged is a constant and the empty product of
$h$-factors is $1$, so the two quantities in the absolute value
are equal. We henceforth assume $1\le j\le k$.

\smallskip
\noindent\emph{Step 1: normalization and the diagonal terms.}
For $\bm r\in\cR_j$, put
\[
 n(\bm r)=r_1\cdots r_j,\qquad
 \Psi_{\bm r}(t)=\prod_{i=1}^j\psi_{i,r_i}(t).
\]
Because the coordinates are pairwise coprime, $n(\bm r)$ is
squarefree, divides $P$, and is at most $D$. Also
\[
 h(n(\bm r))=\prod_{i=1}^jh(r_i)>0.
\]
Every prime dividing $n(\bm r)$ appears in exactly one coordinate.
Independence across primes and \eqref{eq:local-moments} give
\[
 \E\Psi_{\bm r}^2
 =\prod_{p\mid n(\bm r)}h(p)=h(n(\bm r)).
\]
Introduce normalized functions and coefficients
\[
 \Phi_{\bm r}(t)=\frac{\Psi_{\bm r}(t)}{\sqrt{h(n(\bm r))}},
 \qquad v_{\bm r}=c_{\bm r}\sqrt{h(n(\bm r))}.
\]
Then $\E\Phi_{\bm r}^2=1$. The sum to be squared becomes
$\sum_{\bm r}v_{\bm r}\Phi_{\bm r}$, with diagonal contribution
\[
 I:=\sum_{\bm r\in\cR_j}v_{\bm r}^2
   =\sum_{\bm r\in\cR_j}c_{\bm r}^2\prod_{i=1}^jh(r_i).
\]

Let
\[
 G(\bm r,\bm s)=\E(\Phi_{\bm r}\Phi_{\bm s})
\]
be the Gram matrix of these normalized functions. It is real
and symmetric, with diagonal entries $1$. We bound the sum of
absolute values of the off-diagonal entries in each row.

\smallskip
\noindent\emph{Step 2: orthogonality for different coordinate products.}
Suppose $n(\bm r)\ne n(\bm s)$. Since both products are
squarefree, some prime divides exactly one of them. At that prime,
the product $\Psi_{\bm r}\Psi_{\bm s}$ contains exactly one local
factor $\psi_{i,p}$. Its local mean is zero by
\eqref{eq:local-moments}. Factoring the complete average over
primes therefore gives
\[
 G(\bm r,\bm s)=0
 \qquad\text{if }n(\bm r)\ne n(\bm s).
\]
Hence the matrix splits into blocks indexed by the coordinate product.

\smallskip
\noindent\emph{Step 3: the entries within one block.}
Fix a squarefree $n\mid P$ with $n\le D$. For a tuple $\bm r=(r_1,\ldots,r_j)\in\cR_j$ with $r_1 \cdots r_j = n$, let $\iota_{\bm r}(p)$ denote the unique coordinate containing
$p$. Then
\[
 \Phi_{\bm r}(t)
 =\prod_{p\mid n}
             \frac{\psi_{\iota_{\bm r}(p),p}(t)}{\sqrt{h(p)}}.
\]
For two tuples with this product, independence gives
\[
 G(\bm r,\bm s)
 =\prod_{p\mid n}
   \frac{\E_p(\psi_{\iota_{\bm r}(p),p}
                    \psi_{\iota_{\bm s}(p),p})}{h(p)}.
\]
By \eqref{eq:local-moments}, if the assignments agree at $p$, the numerator is $h(p)$; if
they differ, it is $-h(p)^2$. Hence
\[
 G(\bm r,\bm s)
 =\prod_{\substack{p\mid n\\
                \iota_{\bm r}(p)\ne\iota_{\bm s}(p)}}(-h(p)).
\]
The empty product is $1$, in agreement with the diagonal entry.

Fix a tuple $\bm r=(r_1,\ldots,r_j)\in\cR_j$ satisfying
\[
 r_1\cdots r_j=n.
\]
To sum $|G(\bm r,\bm s)|$ over $\bm s\in\cR_j$ with
$n(\bm s)=n$, we describe $\bm s$ by its assignments of primes
to coordinates. We put
\[
 \mathcal P(n)=\{p\in\PP:p\mid n\}.
\]
For each $p\in\mathcal P(n)$, the index
$\iota_{\bm r}(p)\in\{1,\ldots,j\}$ is fixed. We define
\[
 \gamma_p(\ell)=
 \begin{cases}
  1,&\ell=\iota_{\bm r}(p),\\[1mm]
  h(p),&\ell\ne\iota_{\bm r}(p),
 \end{cases}
 \qquad 1\le\ell\le j.
\]
Taking absolute values in the preceding formula for
$G(\bm r,\bm s)$, and using $h(p)>0$, gives
\[
 |G(\bm r,\bm s)|
 =\prod_{p\mid n}\gamma_p\bigl(\iota_{\bm s}(p)\bigr)
 \qquad\bigl(n(\bm s)=n\bigr).
\]
Thus a prime assigned to the same coordinate in $\bm r$
and $\bm s$ contributes $1$, whereas a prime assigned to
different coordinates contributes $h(p)$.

Tuples with product $n$ correspond to maps on $\mathcal P(n)$.
Indeed, let
\[
 \varphi:\mathcal P(n)\longrightarrow\{1,\ldots,j\}
\]
be any map, and define
\[
 s_\ell=\prod_{\substack{p\mid n\\\varphi(p)=\ell}}p
 \qquad(1\le\ell\le j),
\]
where an empty product is $1$. Every coordinate $s_\ell$
divides $n$ and therefore divides $P$. Each prime divisor
of $n$ occurs in exactly one coordinate, namely the coordinate
indexed by $\varphi(p)$. Consequently the coordinates are
pairwise coprime and
\[
 s_1\cdots s_j=\prod_{p\mid n}p=n\le D.
\]
Here we used the squarefreeness of $n$. These properties show
that $\bm s=(s_1,\ldots,s_j)$ belongs to $\cR_j$ and has
$n(\bm s)=n$.

Conversely, if $\bm s\in\cR_j$ and $s_1\cdots s_j=n$,
each prime $p\mid n$ belongs to a unique coordinate, indexed by
$\iota_{\bm s}(p)$. The map $\varphi(p)=\iota_{\bm s}(p)$
then recovers the tuple by the product formula. Starting instead
from a map $\varphi$, the resulting tuple satisfies
$\iota_{\bm s}(p)=\varphi(p)$ for every $p\mid n$.
This proves the bijection
\[
 \{\bm s\in\cR_j:n(\bm s)=n\}
 \longleftrightarrow
 \{\varphi:\mathcal P(n)\to\{1,\ldots,j\}\}.
\]

All assignments are allowed, since the coordinate product remains
$n\le D$. The choices at distinct primes are therefore unrestricted,
and the sum factors as
\[
 \begin{aligned}
 \sum_{\substack{\bm s\in\cR_j\\n(\bm s)=n}}
       |G(\bm r,\bm s)|
 &=\sum_{\varphi:\mathcal P(n)\to\{1,\ldots,j\}}
       \prod_{p\mid n}\gamma_p\bigl(\varphi(p)\bigr)\\
 &=\prod_{p\mid n}
       \left(\sum_{\ell=1}^j\gamma_p(\ell)\right)\\
 &=\prod_{p\mid n}\bigl(1+(j-1)h(p)\bigr).
 \end{aligned}
\]

The diagonal tuple $\bm s=\bm r$ corresponds to
$\varphi(p)=\iota_{\bm r}(p)$ for every $p\mid n$, and contributes
$|G(\bm r,\bm r)|=1$. By Step~2, entries with $n(\bm s)\ne n$
vanish. Subtracting the diagonal contribution gives the full
off-diagonal row sum:
\[
 \sum_{\substack{\bm s\in\cR_j\\\bm s\ne\bm r}}
       |G(\bm r,\bm s)|
 =\prod_{p\mid n}\bigl(1+(j-1)h(p)\bigr)-1.
\]
For $n=1$ this formula also holds: the product on the right is $1$,
and the off-diagonal row sum is zero. If $j=1$, there is
likewise only one tuple with product $n$, namely $(n)$,
and the same formula again gives zero.

\smallskip
\noindent\emph{Step 4: bounding the row sums.}
For any $n\mid P$ with $n\le D$, we have $h(p)\le2m/p$ and
$\omega(n)\le\log D/\log z$. Therefore
\[
 \begin{aligned}
 \prod_{p\mid n}\bigl(1+(j-1)h(p)\bigr)-1
 &\le\exp\left((j-1)\sum_{p\mid n}h(p)\right)-1\\
 &\le\exp\left(\frac{2mk\omega(n)}z\right)-1
 \le\exp\left(\frac{2mk\log D}{z\log z}\right)-1.
 \end{aligned}
\]
The exponent is at most
\[
 \frac{m}{2\alpha_f}\frac{x^{2-\alpha_f}}{\log x},
\]
because $k\le x$, $\log D=x/4$, and $z=x^{\alpha_f}$.
It tends to zero uniformly in $j$ and $k$. For $0\le v\le1$,
the identity $e^v-1=\int_0^v e^s\,ds$ gives $e^v-1\le ev$.
Consequently, for all sufficiently large $x$,
\[
 \delta_x:=\max_{\bm r\in\cR_j}
       \sum_{\substack{\bm s\in\cR_j\\\bm s\ne\bm r}}
                         |G(\bm r,\bm s)|
 \ll_f\frac{k\log D}{z\log z}
 \ll_f\frac{k\log D}{z}.
\]

\smallskip
\noindent\emph{Step 5: from row sums to the quadratic form.}
Subtracting the diagonal contribution $I$ and using
$2|ab|\le a^2+b^2$, we find
\[
 \begin{aligned}
 \left|\E\left(\sum_{\bm r}v_{\bm r}\Phi_{\bm r}\right)^2-I\right|
 &\le\sum_{\bm r\ne\bm s}
             |G(\bm r,\bm s)|\,|v_{\bm r}v_{\bm s}|\\
 &\le\frac12\sum_{\bm r\ne\bm s}
             |G(\bm r,\bm s)|(v_{\bm r}^2+v_{\bm s}^2)\\
 &=\sum_{\bm r}v_{\bm r}^2
       \sum_{\bm s\ne\bm r}|G(\bm r,\bm s)|\\
 &\le\delta_x\sum_{\bm r}v_{\bm r}^2
 \ll_f\frac{k\log D}{z}\,I.
 \end{aligned}
\]
In the third line we used symmetry: the two sums
coming from $v_{\bm r}^2$ and $v_{\bm s}^2$ are equal after
interchanging the indices. Substituting the definitions of $I$,
$v_{\bm r}$, and $\Phi_{\bm r}$ proves \eqref{eq:diagonal}.
\end{proof}

\begin{lemma}[The total weight]\label{lem:S1}
Uniformly under the assumptions of Lemma~\ref{lem:translate},
\begin{equation}\label{eq:S1-asymptotic}
 S_1=(1+o_f(1))A^k.
\end{equation}
\end{lemma}

\begin{proof}
We begin with the diagonal contribution to $\E R^2$.
By \eqref{eq:weight},
\[
 R(t)=\sum_{\bm r\in\cR_k}c_{\bm r}\Psi_{\bm r}(t),
 \qquad
 c_{\bm r}=\prod_{i=1}^k a(r_i),
 \qquad
 \Psi_{\bm r}(t)=\prod_{i=1}^k\psi_{i,r_i}(t).
\]
Expanding the square and interchanging the finite sum with
the complete average gives
\[
 \begin{aligned}
 \E R^2
 &=\E\left(
   \sum_{\bm r\in\cR_k}c_{\bm r}\Psi_{\bm r}
   \sum_{\bm s\in\cR_k}c_{\bm s}\Psi_{\bm s}
   \right)\\
 &=\sum_{\bm r,\bm s\in\cR_k}
   c_{\bm r}c_{\bm s}\,
   \E(\Psi_{\bm r}\Psi_{\bm s}).
 \end{aligned}
\]
The total contribution of the diagonal terms is
\[
 \sum_{\bm r\in\cR_k}
 c_{\bm r}^{\,2}\E\Psi_{\bm r}^{\,2}.
\]
To evaluate it, fix $\bm r=(r_1,\ldots,r_k)\in\cR_k$ and put
$n=r_1\cdots r_k$. Each $r_i$ is a squarefree divisor
of $P$, and the coordinates are pairwise coprime.
Consequently, every prime divisor of $n$ occurs in
exactly one coordinate. By the definition of
$\psi_{i,r_i}$,
\[
 \Psi_{\bm r}(t)^2
 =\prod_{i=1}^k\prod_{p\mid r_i}\psi_{i,p}(t)^2.
\]
The prime moduli are distinct, so the Chinese remainder theorem
factors the complete average into local averages. By
$\E_p\psi_{i,p}^2=h(p)$ in \eqref{eq:local-moments},
\[
 \begin{aligned}
 \E\Psi_{\bm r}^{\,2}
 &=\prod_{i=1}^k\prod_{p\mid r_i}
                  \E_p\psi_{i,p}^{\,2}\\
 &=\prod_{i=1}^k\prod_{p\mid r_i}h(p)\\
 &=\prod_{i=1}^k h(r_i).
 \end{aligned}
\]
Coprimality ensures that each prime occurs at only one position,
so this calculation uses only independence across prime moduli.

Since
\[
 c_{\bm r}^{\,2}=\prod_{i=1}^k a(r_i)^2,
\]
the diagonal contribution is thus
\[
 \Delta
 :=\sum_{\bm r\in\cR_k}
       c_{\bm r}^{\,2}\E\Psi_{\bm r}^{\,2}
 =\sum_{\bm r\in\cR_k}
       \prod_{i=1}^k a(r_i)^2h(r_i).
\]
We call this nonnegative sum the diagonal mass of $R$.

The complete second moment is therefore
\[
 \E R^2
 =\Delta+
   \sum_{\substack{\bm r,\bm s\in\cR_k\\
                    \bm r\ne\bm s}}
       c_{\bm r}c_{\bm s}\,
       \E(\Psi_{\bm r}\Psi_{\bm s}).
\]
Lemma~\ref{lem:diagonal} bounds the second sum in terms of
$\Delta$. To estimate $\Delta$, first remove the coprimality and
product restrictions. The coordinates can then be summed separately,
and the definition of $A$ in \eqref{eq:coefficients} gives
\[
 \begin{aligned}
 \sum_{r_1\mid P}\cdots\sum_{r_k\mid P}
                       \prod_{i=1}^ka(r_i)^2h(r_i)
 &=\prod_{i=1}^k\left(\sum_{r_i\mid P}a(r_i)^2h(r_i)\right)
 =A^k.
 \end{aligned}
\]
Nonnegativity gives $0\le\Delta\le A^k$. A tuple missing from
$\Delta$ either has product greater than $D$ or has two coordinates
sharing a prime factor. We bound these contributions separately;
counting a tuple in both only increases the bound.

\smallskip
\noindent\emph{The mass beyond the product cutoff.}
If $r_1\cdots r_k>D$, then
$1\le(r_1\cdots r_k/D)^\beta$. Multiplying by this factor,
and then dropping the product restriction in a nonnegative sum,
gives
\[
 \begin{aligned}
 &\sum_{\substack{r_i\mid P\ (1\le i\le k)\\r_1\cdots r_k>D}}
                  \prod_{i=1}^ka(r_i)^2h(r_i)\\
 &\qquad\le D^{-\beta}
       \sum_{r_1\mid P}\cdots\sum_{r_k\mid P}
                  \prod_{i=1}^ka(r_i)^2h(r_i)r_i^\beta
 =D^{-\beta}A_\beta^k.
 \end{aligned}
\]
This is Rankin's bound.
Since $A_\beta\le A+C\beta$ and $A\ge1$, we obtain
\[
 A_\beta^k
 \le A^k\left(1+\frac{C\beta}{A}\right)^k
 \le A^k\exp(Ck\beta).
\]
Also $D^{-\beta}=e^{-u}$, and \eqref{eq:budget-margin} gives
$Ck\beta\le u/4$. In particular,
\begin{equation}\label{eq:tail}
 \sum_{\substack{r_i\mid P\ (1\le i\le k)\\r_1\cdots r_k>D}}
                       \prod_i a(r_i)^2h(r_i)
 \le D^{-\beta}A_\beta^k
 \le A^k\exp(-u+Ck\beta)
 \le A^ke^{-u/2}.
\end{equation}

\smallskip
\noindent\emph{The mass lost through noncoprime coordinates.}
For $k=1$ there is no coprimality restriction, so this contribution
is zero. Assume $k\ge2$. For each $p\mid P$, put
\[
 W_p=\sum_{\substack{r\mid P\\p\mid r}}a(r)^2h(r).
\]
The bound $|a(r)|\le1$ implies $a(r)^2\le|a(r)|$. Hence
\eqref{eq:coefficients-prime} yields
\[
 0\le W_p\le\sum_{\substack{r\mid P\\p\mid r}}|a(r)|h(r)
 \ll_f\frac1p.
\]
Fix a pair $i<j$ and a prime $p\mid P$. Among all unrestricted
tuples of divisors of $P$, the total mass with $p\mid r_i$
and $p\mid r_j$ is
\[
 \begin{aligned}
 &\sum_{\substack{r_\ell\mid P\ (1\le\ell\le k)\\p\mid r_i,\ p\mid r_j}}
                      \prod_{\ell=1}^k a(r_\ell)^2h(r_\ell)\\
 &\qquad=
 \left(\sum_{\substack{r_i\mid P\\p\mid r_i}}a(r_i)^2h(r_i)\right)
 \left(\sum_{\substack{r_j\mid P\\p\mid r_j}}a(r_j)^2h(r_j)\right)
 \prod_{\ell\ne i,j}\left(\sum_{r_\ell\mid P}a(r_\ell)^2h(r_\ell)\right)\\
 &\qquad=W_p^2A^{k-2}.
 \end{aligned}
\]
For $k=2$, the last product is empty and equals $1$.
Every tuple that is not pairwise coprime is counted for at least
one such pair and prime. The union bound therefore gives
\[
 \begin{aligned}
 &\sum_{\substack{r_\ell\mid P\ (1\le\ell\le k)\\
                  (r_i,r_j)>1\ \text{for some }i<j}}
                    \prod_{\ell=1}^k a(r_\ell)^2h(r_\ell)\le 
                    \sum_{1\le i<j\le k}\sum_{p|P}
                    \sum_{\substack{r_\ell\mid P\ (1\le\ell\le k)\\p\mid r_i,\ p\mid r_j}}
                      \prod_{\ell=1}^k a(r_\ell)^2h(r_\ell)\\
 &\qquad\le\binom{k}{2}A^{k-2}\sum_{p\mid P}W_p^2
 \ll_f k^2A^{k-2}\sum_{p\mid P}\frac1{p^2}
 \ll_f\frac{k^2}{z}A^k.
 \end{aligned}
\]
For completeness, the last estimate follows from $A\ge1$ and
\[
 \sum_{p\mid P}\frac1{p^2}
 \le\sum_{n>z}\frac1{n^2}
 \le\int_{\lfloor z\rfloor}^\infty\frac{dt}{t^2}
 =\frac1{\lfloor z\rfloor}\le\frac2z
 \qquad(z\ge2).
\]

\smallskip
\noindent\emph{Combining the diagonal and averaging estimates.}
Combining the cutoff estimate \eqref{eq:tail} with the preceding
bound for noncoprime coordinates proves
\[
 0\le A^k-\Delta\ll_f
 A^k\left(e^{-u/2}+\frac{k^2}{z}\right),
\]
or equivalently
\begin{equation}\label{eq:delta-estimate}
 \Delta=A^k\left(1+O_f\left(e^{-u/2}+\frac{k^2}{z}\right)\right).
\end{equation}
Apply Lemma~\ref{lem:diagonal} with $j=k$ and
$c_{\bm r}=\prod_{i=1}^ka(r_i)$. The function in that lemma
is then $R$, and its diagonal sum is precisely $\Delta$.
Since $\Delta\le A^k$, we have
\[
 \E R^2=\Delta+O_f\left(\frac{k\log D}{z}\Delta\right)
 =\Delta+O_f\left(\frac{k\log D}{z}A^k\right).
\]
Combining this with \eqref{eq:delta-estimate} and
\eqref{eq:averaging-error} gives
\[
 S_1=A^k\left(1+O_f\left(
 e^{-u/2}+\frac{k^2}{z}+\frac{k\log D}{z}\right)\right)
       +O(\mathcal E_x).
\]
All errors are small relative to $A^k$. Indeed,
\[
 e^{-u/2}=\exp\left(-\frac{(\log x)^5}{8}\right)\longrightarrow0,
 \qquad
 \frac{k^2}{z}\le x^{2-\alpha_f}\longrightarrow0,
 \qquad
 \frac{k\log D}{z}\le\frac14 x^{2-\alpha_f}\longrightarrow0.
\]
Finally $A^k\ge1$, so the additive error satisfies
$\mathcal E_x/A^k\le\mathcal E_x=o_f(1)$.
These bounds use only $k\le x$ and the fixed polynomial $f$,
so they hold uniformly in $H$, $S$, and $b$. This proves
\eqref{eq:S1-asymptotic}. In particular, $S_1\ge A^k/2>0$
for all sufficiently large $x$.
\end{proof}

\subsection{The contribution of uncovered positions}

\begin{lemma}\label{lem:S2}
Uniformly under the assumptions of Lemma~\ref{lem:translate},
\begin{equation}\label{eq:S2-bound}
 S_2\ll_f A^k k\left(e^{-u}+\frac{(\log D)^2}{z^2}\right)
                  +k\mathcal E_x.
\end{equation}
\end{lemma}

\begin{proof}
By symmetry, it suffices to bound the contribution of position $k$.
All preceding estimates are unchanged by permuting the positions.

\smallskip
\noindent\emph{Step 1: replacing the factors of an uncovered position.}
For $v\mid P$ with $v\le D$, define
\[
 \epsilon(v)=
 \sum_{\substack{r\mid P\\(r,v)=1\\r\le D/v}}a(r)h(r).
\]
Set
\begin{equation}\label{eq:Rk-definition}
 R_k(t)=\sum_{\bm r\in\cR_{k-1}}
       \epsilon(r_1\cdots r_{k-1})
       \prod_{j=1}^{k-1}a(r_j)\psi_{j,r_j}(t).
\end{equation}
If $k=1$, the sum is over
the empty tuple and means the constant $R_1(t)=\epsilon(1)$.

Fix the first $k-1$ coordinates and put $v=r_1\cdots r_{k-1}$.
The last coordinate ranges over divisors $r\mid P$ with $(r,v)=1$
and $r\le D/v$: these are precisely the coprimality and product
conditions. Every tuple in $\cR_{k-1}$ admits the last coordinate
$1$. Grouping the terms of $R$ by their first $k-1$ coordinates
therefore gives
\[
 \begin{aligned}
 R(t)=\sum_{\bm r\in\cR_{k-1}}
 \left(\prod_{j=1}^{k-1}a(r_j)\psi_{j,r_j}(t)\right)
 \sum_{\substack{r\mid P\\(r,v)=1\\r\le D/v}}a(r)\psi_{k,r}(t),
 \qquad v=r_1\cdots r_{k-1}.
 \end{aligned}
\] Also, 
\[
 \begin{aligned}
 R_k(t)=\sum_{\bm r\in\cR_{k-1}}
 \left(\prod_{j=1}^{k-1}a(r_j)\psi_{j,r_j}(t)\right)
 \sum_{\substack{r\mid P\\(r,v)=1\\r\le D/v}}a(r)h(r),
 \qquad v=r_1\cdots r_{k-1}.
 \end{aligned}
\]If position $k$ is uncovered (that is $(F_{k}(t), P)=1$), then $\psi_{k,r}(t)=h(r)$ for
every $r\mid P$. The inner sums are equal, so $R(t)=R_k(t)$. We obtain
\begin{equation}\label{eq:uncovered-majorant}
 R(t)^2\ind_{(F_k(t),P)=1}\le R_k(t)^2.
\end{equation}
When the indicator is $1$ there is equality, and when it is $0$
the left-hand side is zero.

\smallskip
\noindent\emph{Step 2: averaging the new square.}
To compare the finite and complete averages of $R_k^2$,
regard the expansion of $R$ as a polynomial in the prime
divisibility indicators. Setting all indicators
$\ind_{p\mid F_k(t)}$, $p\mid P$, equal to zero replaces
$\psi_{k,p}(t)$ by $h(p)$ and hence $\psi_{k,r}(t)$ by $h(r)$ for every $r|P$.
By Step~1, the resulting expression is exactly $R_k$.
This is an algebraic substitution and does not assert the
existence of an integer $t$ realizing prescribed indicators.

In \eqref{eq:lambda-expansion}, the substitution deletes every
term with $d_k>1$, because such a term contains at least one
prime indicator for position $k$. If $d_k=1$, its last-coordinate
indicator is the constant $1$, and the term remains unchanged.
Thus
\[
 R_k(t)=\sum_{\bm d\in\cR_{k-1}}
       \lambda_{(d_1,\ldots,d_{k-1},1)}
       \prod_{j=1}^{k-1}\ind_{d_j\mid F_j(t)}.
\]
Appending a coordinate $1$ identifies $\cR_{k-1}$ with the
subset of $\cR_k$ having last coordinate $1$. Therefore
\[
 \sum_{\bm d\in\cR_{k-1}}
             |\lambda_{(d_1,\ldots,d_{k-1},1)}|
 \le\sum_{\bm d\in\cR_k}|\lambda_{\bm d}|\le D^{3/2}.
\]
Every retained coordinate product is still at most $D$, and the
residue-class count in Subsection~\ref{subsec:finite-averages}
applies to each pair of retained terms. Consequently, by \eqref{SUM_R_SUBSET}
\begin{equation}\label{eq:Rk-average}
 \begin{aligned}
 \left|\frac1T\sum_{t=1}^TR_k(t)^2-\E R_k^2\right|
 &\le\frac{D^{2\log m/\log z}}T
     \left(\sum_{\bm d\in\cR_{k-1}}
                 |\lambda_{(d_1,\ldots,d_{k-1},1)}|\right)^2\\
 &\le\mathcal E_x.
 \end{aligned}
\end{equation}
We have thus avoided averaging the uncovered-position indicator,
which may have period $P$.

\smallskip
\noindent\emph{Step 3: estimating the remaining coefficient sum.}
Compare $\epsilon(v)$ with the unrestricted sum
\[
 \sum_{r\mid P}a(r)h(r),
\]
over all divisors $r$ of $P$. This is zero by
\eqref{eq:exact-cancellation}.

For a fixed divisor $v\mid P$ with $v\le D$, the sum
\[
 \epsilon(v)
 =
 \sum_{\substack{r\mid P\\(r,v)=1\\r\le D/v}}
       a(r)h(r)
\]
retains the divisors $r$ satisfying $(r,v)=1$ and $r\le D/v$.
The omitted divisors satisfy
\[
 r>D/v
 \qquad\text{or}\qquad
 (r,v)>1.
\]
Partitioning the divisors of $P$ into retained and omitted terms gives
\[
 \begin{aligned}
 0
 &=\sum_{r\mid P}a(r)h(r)\\
 &=\sum_{\substack{r\mid P\\(r,v)=1\\r\le D/v}}
       a(r)h(r)
   +\sum_{\substack{r\mid P\\
          r>D/v\ \text{or}\ (r,v)>1}}
       a(r)h(r)\\
 &=\epsilon(v)
   +\sum_{\substack{r\mid P\\
          r>D/v\ \text{or}\ (r,v)>1}}
       a(r)h(r).
 \end{aligned}
\]
Here an omitted divisor satisfying both exclusion
conditions is still counted only once.
Rearranging the identity yields
\[
 \epsilon(v)
 =-\sum_{\substack{r\mid P\\
          r>D/v\ \text{or}\ (r,v)>1}}
       a(r)h(r).
\]

Every omitted term is negative. Indeed, $r=1$ is retained, since
\[
 (1,v)=1,
 \qquad
 1\le D/v.
\]
The second inequality follows from $v\le D$.
Thus every omitted divisor satisfies $r>1$, and hence
\[
 a(r)=-\frac1{B\log r}<0.
\]
Since $h(r)>0$ for every $r\mid P$, we conclude that
\[
 \epsilon(v)
 =
 \sum_{\substack{r\mid P\\
          r>D/v\ \text{or}\ (r,v)>1}}
       |a(r)|h(r)
 \ge0.
\]
An omitted divisor either satisfies $r>D/v$ or has a prime
$p\mid v$ dividing $r$. Nonnegativity and the union bound give
\[
 \epsilon(v)
 \le\sum_{\substack{r\mid P\\r>D/v}}|a(r)|h(r)
    +\sum_{p\mid v}\sum_{\substack{r\mid P\\p\mid r}}|a(r)|h(r).
\]

For the first sum, $r>D/v$ implies $1\le(vr/D)^\beta$.
Since $D/v\ge1$, none of its terms has $r=1$. Hence
\[
 \begin{aligned}
 \sum_{\substack{r\mid P\\r>D/v}}|a(r)|h(r)
 &\le\left(\frac vD\right)^\beta
          \sum_{\substack{r\mid P\\r>D/v}}|a(r)|h(r)r^\beta\\
 &\le\left(\frac vD\right)^\beta
          \sum_{\substack{r\mid P\\r>1}}|a(r)|h(r)r^\beta
 \le C\left(\frac vD\right)^\beta,
 \end{aligned}
\]
by \eqref{eq:coefficients-moments} with $\theta=\beta$.
For the second sum, \eqref{eq:coefficients-prime} gives
\[
 \sum_{p\mid v}\sum_{\substack{r\mid P\\p\mid r}}|a(r)|h(r)
 \ll_f\sum_{p\mid v}\frac1p.
\]
Every prime factor of $v$ exceeds $z$, and $v\le D$, so
\[
 \sum_{p\mid v}\frac1p
 \le\frac{\omega(v)}z
 \le\frac{\log v}{z\log z}
 \le\frac{\log D}{z\log z}.
\]
For $v=1$, the sum is empty and $\omega(v)=\log v=0$, so the
bound still holds. Combining the estimates, we obtain
\begin{equation}\label{eq:epsilon}
 0\le\epsilon(v)
 \ll_f\left(\frac vD\right)^\beta+\frac{\log D}{z\log z}
 \ll_f\left(\frac vD\right)^\beta+\frac{\log D}{z}.
\end{equation}
Applying $(a+b)^2\le2a^2+2b^2$ to \eqref{eq:epsilon} gives
\begin{equation}\label{EPS_v_^2}
 \epsilon(v)^2\ll_f
 D^{-2\beta}v^{2\beta}+\frac{(\log D)^2}{z^2}.
\end{equation}

\smallskip
\noindent\emph{Step 4: the complete second moment of $R_k$.}
The diagonal sum corresponding to \eqref{eq:Rk-definition} is
\[
 I_k=\sum_{\bm r\in\cR_{k-1}}
       \epsilon(v)^2\prod_{j=1}^{k-1}a(r_j)^2h(r_j),
 \qquad v=r_1\cdots r_{k-1}.
\]
For each tuple in this sum, $v\mid P$ and $v\le D$, so the
preceding estimate applies. Since
$v^{2\beta}=\prod_{j=1}^{k-1}r_j^{2\beta}$, it gives
\[
 \begin{aligned}
 I_k
 &\ll_f D^{-2\beta}
       \sum_{\bm r\in\cR_{k-1}}
             \prod_{j=1}^{k-1}a(r_j)^2h(r_j)r_j^{2\beta}\\
 &\qquad+\frac{(\log D)^2}{z^2}
       \sum_{\bm r\in\cR_{k-1}}
             \prod_{j=1}^{k-1}a(r_j)^2h(r_j).
 \end{aligned}
\]
Both sums have nonnegative terms. We may therefore remove the
coprimality and product restrictions while retaining $r_j\mid P$.
For $\theta\in\{0,2\beta\}$ this gives
\[
 \begin{aligned}
 \sum_{\bm r\in\cR_{k-1}}
             \prod_{j=1}^{k-1}a(r_j)^2h(r_j)r_j^\theta
 &\le\sum_{r_1\mid P}\cdots\sum_{r_{k-1}\mid P}
             \prod_{j=1}^{k-1}a(r_j)^2h(r_j)r_j^\theta\\
 &=\prod_{j=1}^{k-1}
       \left(\sum_{r_j\mid P}a(r_j)^2h(r_j)r_j^\theta\right)
 =A_\theta^{k-1}.
 \end{aligned}
\]
Consequently,
\[
 I_k\ll_f D^{-2\beta}A_{2\beta}^{k-1}
             +\frac{(\log D)^2}{z^2}A^{k-1}.
\]
When $k=1$, there is one empty tuple and $I_1=\epsilon(1)^2$.
The estimate \eqref{EPS_v_^2} for $\epsilon(v)^2$ at $v=1$ gives the same bound,
since $A_{2\beta}^0=A^0=1$.

Apply Lemma~\ref{lem:diagonal} with $j=k-1$ and
\[
 c_{\bm r}=\epsilon(r_1\cdots r_{k-1})
                         \prod_{j=1}^{k-1}a(r_j).
\]
The corresponding function is $R_k$, and its diagonal sum is $I_k$.
It follows that
\[
 \E R_k^2
 \le\left(1+O_f\left(\frac{k\log D}{z}\right)\right)I_k
 \ll_f I_k,
\]
where the multiplier is bounded for large $x$, because
$k\log D/z\le\frac14 x^{2-\alpha_f}\to0$.
The $j=0$ case of the lemma covers $k=1$ exactly.

To bound the term involving $A_{2\beta}$, use
$A_{2\beta}\le A+2C\beta$ and $A\ge1$:
\[
 \begin{aligned}
 D^{-2\beta}A_{2\beta}^{k-1}
 &\le e^{-2u}A^{k-1}
                  \left(1+\frac{2C\beta}{A}\right)^{k-1}\\
 &\le A^{k-1}\exp\bigl(-2u+2C(k-1)\beta\bigr)\\
 &\le A^k\exp(-2u+2Ck\beta)
 \le A^ke^{-u}.
 \end{aligned}
\]
Here \eqref{eq:budget-margin} gives $2Ck\beta\le u/2$,
and we used the weaker bound $2Ck\beta\le u$.
The remaining term is at most $(\log D)^2A^k/z^2$, since
$A^{k-1}\le A^k$. Thus
\begin{equation}\label{eq:Rk-complete-bound}
 \E R_k^2\ll_f A^k\left(e^{-u}+\frac{(\log D)^2}{z^2}\right).
\end{equation}

\smallskip
\noindent\emph{Step 5: summing over the uncovered positions.}
Combining the pointwise majorant \eqref{eq:uncovered-majorant}
with \eqref{eq:Rk-average} and \eqref{eq:Rk-complete-bound}, we get
\[
 \begin{aligned}
 \frac1T\sum_{t=1}^Tw(t)\ind_{(F_k(t),P)=1}
 &\le\frac1T\sum_{t=1}^TR_k(t)^2\\
 &\le\E R_k^2+\mathcal E_x\\
 &\ll_f A^k\left(e^{-u}+\frac{(\log D)^2}{z^2}\right)
                 +\mathcal E_x.
 \end{aligned}
\]
For any index $i$, relabel the positions so that $i$ is last.
The allowed tuples, coefficients, local moments, and averaging error
are unchanged, so the bound holds with the same constant for every
$i$. Interchanging the finite sums defining $S_2$ gives
\[
 \begin{aligned}
 S_2
 &=\sum_{i=1}^k\frac1T\sum_{t=1}^Tw(t)\ind_{(F_i(t),P)=1}\\
 &\ll_f kA^k\left(e^{-u}+\frac{(\log D)^2}{z^2}\right)
                 +k\mathcal E_x,
 \end{aligned}
\]
which is \eqref{eq:S2-bound}.
\end{proof}

\begin{proof}[Completion of the proof of Lemma~\ref{lem:translate}]
For all sufficiently large $x$, Lemma~\ref{lem:S1} gives
$S_1\ge A^k/2>0$. Dividing the estimate in Lemma~\ref{lem:S2}
by this lower bound, and then using $A^k\ge1$, gives
\[
 \begin{aligned}
 \frac{S_2}{S_1}
 &\ll_f k e^{-u}+\frac{k(\log D)^2}{z^2}
                   +\frac{k\mathcal E_x}{A^k}\\
 &\le O_f(1)\left(k e^{-u}
                   +\frac{k(\log D)^2}{z^2}+k\mathcal E_x\right).
 \end{aligned}
\]
Substituting $k\le x$, $u=(\log x)^5/4$, $\log D=x/4$,
and $z=x^{\alpha_f}$, we obtain
\[
 \frac{S_2}{S_1}
 \ll_f x\exp\left(-\frac{(\log x)^5}{4}\right)
       +x^{3-2\alpha_f}
       +x\exp\bigl(-x/4+o_f(x)\bigr)
 \longrightarrow0.
\]
Indeed, the logarithm of the first term is
$\log x-(\log x)^5/4\to-\infty$; the second term is at most
$x^{-11}$ because $\alpha_f\ge7$; and the third term is at most
$xe^{-x/8}$ for all sufficiently large $x$.
In particular, $S_2<S_1$.

Write
\[
 U(t)=\sum_{i=1}^k\ind_{(F_i(t),P)=1}.
\]
Then
\[
 0<S_1-S_2=\frac1T\sum_{t=1}^Tw(t)(1-U(t)).
\]
At least one summand is positive. For that integer $t$ we have
$w(t)>0$ and $U(t)<1$. As $U(t)$ is a nonnegative integer,
it must equal zero. Thus $(F_i(t),P)>1$ for every $i$.
Each such greatest common divisor has a prime factor $p\mid P$,
and this prime divides $F_i(t)=g(b+Qt+s_i)$.
By the definition of $P$, it satisfies $z<p\le y$ and $p>p_0$.
Formula \eqref{eq:divisibility-equivalence} therefore gives
$p\mid f(b+Qt+s_i)$, with $1\le t\le T\le e^x$, as required
by Lemma~\ref{lem:translate}.

\end{proof}

\section{Producing an interval of composite values}

\begin{proof}[Proof of Theorem~\ref{thm:main}]
Choose a fixed $C_f\ge1$ so that, for all sufficiently large $x$,
\[
 \sigma_f(x)\le\frac{C_f}{\log x}.
\]
Such a choice is possible by \eqref{eq:polynomial-mertens}.
Let $\eta_f$ be the constant from Lemma~\ref{lem:translate}, and set
\begin{equation}\label{eq:initial-length}
 \varepsilon_f=\frac{\eta_f}{4C_f},\qquad
 H=\lfloor\varepsilon_f x\log x\rfloor,
 \qquad Q=Q(x).
\end{equation}
Both $\varepsilon_f$ and $C_f$ depend only on $f$.
Since $\varepsilon_f x\log x\to\infty$, we have $H\ge1$ for
large $x$. Also $\varepsilon_f\le\log x$ eventually, so
$H\le\varepsilon_f x\log x\le x(\log x)^2$.
Thus $H$ lies in the range required by the covering lemma.

\smallskip
\noindent\emph{Step 1: covering most positions with the small primes.}
For each integer $b_0$ with $0\le b_0<Q$, define
\[
 S(b_0)=\{s\in[1,H]\cap\cZ:
             p\nmid g(b_0+s)\text{ for every }p_0<p\le x\}.
\]
Fix $s\in[1,H]\cap\cZ$ and count the $Q$ choices of $b_0$
for which $s\in S(b_0)$.

First consider a single prime $p$ with $p_0<p\le x$.
By the definition of $\rho(p)$, the congruence
\[
 g(u)\equiv0\pmod p
\]
has exactly $\rho(p)$ solutions modulo $p$.
The map $c\mapsto c+s$ is a bijection on the residue classes
modulo $p$: each residue $u$ is obtained from the unique residue
$c=u-s$. Hence the congruence
\[
 g(c+s)\equiv0\pmod p
\]
also has exactly $\rho(p)$ solutions modulo $p$.
There are therefore exactly $p-\rho(p)$ choices of
$b_0\bmod p$ for which $p\nmid g(b_0+s)$.

We now impose these conditions simultaneously for all primes
$p_0<p\le x$. Since
\[
 Q=\prod_{p\le x}p,
\]
the Chinese remainder theorem gives a bijection
\[
 \{0,\ldots,Q-1\}
 \longrightarrow
 \prod_{p\le x}\cZ/p\cZ,
 \qquad
 b_0\longmapsto (b_0\bmod p)_{p\le x}.
\]
Thus a choice of $c_p$ at each prime $p\le x$ determines a unique
$b_0\in\{0,\ldots,Q-1\}$ satisfying $b_0\equiv c_p\pmod p$
at every prime.

For each prime $p_0<p\le x$, we may choose $c_p$ from the
$p-\rho(p)$ residues found above. For primes $p\le p_0$,
the definition of $S(b_0)$ imposes no condition, so all $p$
residues are allowed. Thus the number of admissible choices
of $b_0$ is the product of these numbers of choices:
\[
 \#\{0\le b_0<Q:s\in S(b_0)\}
 =
 \left(\prod_{\substack{p\le x\\p\le p_0}}p\right)
 \left(\prod_{p_0<p\le x}(p-\rho(p))\right).
\]
Recall that $\rho(p)=0$ for $p\le p_0$. Consequently,
\[
 \begin{aligned}
 \#\{0\le b_0<Q:s\in S(b_0)\}
 &=\prod_{p\le x}(p-\rho(p))\\
 &=Q\prod_{p\le x}\left(1-\frac{\rho(p)}p\right)\\
 &=Q\,\sigma_f(x).
 \end{aligned}
\]
Equivalently, for every fixed $s\in[1,H]\cap\cZ$,
\[
 \frac1Q\sum_{b_0=0}^{Q-1}
       \mathbf{1}_{\{s\in S(b_0)\}}
 =\sigma_f(x).
\]

Summing over the positions and interchanging the finite sums gives
\[
 \begin{aligned}
 \frac1Q\sum_{b_0=0}^{Q-1}|S(b_0)|
 &=\frac1Q\sum_{b_0=0}^{Q-1}\sum_{s=1}^H
                                      \ind_{s\in S(b_0)}\\
 &=\sum_{s=1}^H\frac1Q\sum_{b_0=0}^{Q-1}
                                      \ind_{s\in S(b_0)}
 =H\sigma_f(x).
 \end{aligned}
\]
 By the choices
of $H$ and $\varepsilon_f$,
\[
 H\sigma_f(x)
 \le\varepsilon_f x\log x\,\frac{C_f}{\log x}
 =\frac{\eta_f x}{4}.
\]
Choose a representative $0\le b_0<Q$ for which the cardinality
does not exceed the average. Writing $S=S(b_0)$, we have
\begin{equation}\label{eq:initial-survivors}
 S=\{1\le s\le H:p\nmid g(b_0+s)
                         \text{ for all }p_0<p\le x\},
 \qquad |S|\le\frac{\eta_f x}{4}\le\eta_f x.
\end{equation}
For every $s\in\{1,\ldots,H\}\setminus S$, some prime $p\in(p_0,x]$ divides
$g(b_0+s)$. Since $p\mid Q$, this divisibility persists after
replacing $b_0$ by any integer congruent to it modulo $Q$.
Indeed, a polynomial with integer coefficients preserves
congruences, so
\[
 n\equiv b_0\pmod Q
 \quad\Longrightarrow\quad
 g(n+s)\equiv g(b_0+s)\equiv0\pmod p.
\]

\smallskip
\noindent\emph{Step 2: choosing the location and covering the survivors.}
For a sufficiently large real $X$, choose the auxiliary parameter
and the starting point by
\begin{equation}\label{eq:location}
 x=\frac13\log X,\qquad
 M=Q\left\lceil\frac{X}{2Q}\right\rceil,
 \qquad b=M+b_0.
\end{equation}
Use this value of $x$ in Step~1. Since $X\to\infty$ implies
$x\to\infty$, all preceding estimates hold once $X$ is large enough.
The definition of $M$ in \eqref{eq:location} and the ceiling
inequalities $v\le\lceil v\rceil<v+1$ give
\[
 \frac X2\le M<\frac X2+Q.
\]
Since $M$ is a multiple of $Q$, we have $b\equiv b_0\pmod Q$.
The integer $b$ may greatly exceed $Q$, but Lemma~\ref{lem:translate}
holds uniformly for every integer $b$.

We distinguish two cases according to the set $S$.

\smallskip
\noindent\emph{Case 1: $S=\varnothing$.}
Set $t=1$ and
\[
 N=b+Q=M+b_0+Q.
\]
For every $1\le s\le H$, the definition of $S$ supplies a prime
$p\in(p_0,x]$ with $p\mid g(b_0+s)$. Since
$N\equiv b_0\pmod Q$, the same prime divides $g(N+s)$.
It does not divide $D_f$, by the choice of $p_0$. As $f(N+s)$
is an integer and $g(N+s)=D_ff(N+s)$, it follows that
$p\mid f(N+s)$, as also follows from
\eqref{eq:divisibility-equivalence}. All positions are therefore
covered by small primes, and Lemma~\ref{lem:translate} is not needed.

\smallskip
\noindent\emph{Case 2: $S\ne\varnothing$.}
We now have $1\le|S|\le\eta_f x$ by
\eqref{eq:initial-survivors}, and $S\subseteq[1,H]\cap\cZ$ with
$1\le H\le x(\log x)^2$. All hypotheses of
Lemma~\ref{lem:translate} are therefore satisfied for the integer
$b$ chosen above. Let $1\le t\le e^x$ be supplied by that lemma,
and put
\[
 N=b+Qt=M+b_0+Qt.
\]
For $s\in S$, the lemma gives a prime divisor $p$ of $f(N+s)$
in $(z,y]$, where $z=x^{\alpha_f}$ and $y$ is given in
\eqref{eq:auxiliary-range}. For $s\in\{1,\ldots,H\}\setminus S$, choose the prime
$p\in(p_0,x]$ from Step~1. Since $N\equiv b_0\pmod Q$,
it divides $g(N+s)$ and hence $f(N+s)$ by
\eqref{eq:divisibility-equivalence}. Thus the second covering
stage preserves all divisibilities established by the first stage.

\smallskip
In either case, we obtain integers $N$ and $t$ satisfying
\[
 N=M+b_0+Qt,\qquad 1\le t\le e^x,
\]
and each $f(N+s)$, $1\le s\le H$, has a prime divisor at most
$y$. This includes the small-prime divisors, which are at most $y$
since $x<z<y$ for large $x$. We now check the location and length
of the interval and show that these divisors are proper.

\smallskip
\noindent\emph{Step 3: verifying that the interval lies in $(X/2,X]$.}
Since $t\ge1$ and $b_0\ge0$,
\[
 N=M+b_0+Qt\ge M+Q>\frac X2.
\]
Thus the lower endpoint $N+1$ is strictly greater than $X/2$.
For the upper endpoint, the bounds on $M$, $b_0$, and $t$ imply
\[
 N+H<\frac X2+2Q+Qe^x+H.
\]
The prime number theorem, in the form
$\sum_{p\le x}\log p=x+o(x)$, yields
\[
 \log Q(x)=\sum_{p\le x}\log p=(1+o(1))x.
\]
Since $X=e^{3x}$, it follows that
\[
 Q=X^{1/3+o(1)},\qquad Qe^x=X^{2/3+o(1)}=o(X).
\]
Also 
\[
 H\le\varepsilon_f x\log x
 =O_f(\log X\log\log X)=o(X).
\]
Thus $2Q+Qe^x+H=o(X)$ and is less than $X/2$ for large $X$.
It follows that $N+H<X$, and hence
\[
 \{N+1,\ldots,N+H\}\subset(X/2,X].
\]

\smallskip
\noindent\emph{Step 4: checking that the prime divisors are proper.}
Let $a>0$ be the leading coefficient of $f$, and recall that
$m=\deg f\ge1$. The lower-degree terms imply
\[
 f(n)=an^m+O_f(n^{m-1}),\qquad
 \frac{f(n)}{n^m}\longrightarrow a
 \quad(n\to+\infty).
\]
Thus there is a threshold depending only on $f$ such that
$f(n)\ge(a/2)n^m$ for every positive integer beyond it.
All arguments $N+s$ exceed $X/2$, so, uniformly for $1\le s\le H$,
\[
 f(N+s)\ge\frac a2(N+s)^m
          \ge\frac{a}{2^{m+1}}X^m>0
\]
once $X$ is sufficiently large.

On the other hand,
\[
 \frac{\log y}{\log X}
 =\frac{x}{(\log x)^5\log X}
 =\frac1{3(\log x)^5}\longrightarrow0.
\]
In particular, for large $X$ this ratio is at most $m/2$, giving
$y\le X^{m/2}$. Since $a/2^{m+1}$ is fixed and positive,
\[
 X^{m/2}<\frac{a}{2^{m+1}}X^m\le f(N+s)
\]
for all sufficiently large $X$. Every supplied prime divisor $p$
therefore satisfies
\[
 2\le p\le y<f(N+s).
\]
Thus $f(N+s)$ is
composite. The argument also applies when $m=1$.

\smallskip
\noindent\emph{Step 5: the length.}
The floor in \eqref{eq:initial-length} changes its argument by
less than $1$. Hence
\[
 \begin{aligned}
 H&=\varepsilon_f x\log x+O(1)\\
  &=\frac{\varepsilon_f}{3}\log X
                       (\log\log X-\log3)+O(1)\\
  &=\left(\frac{\varepsilon_f}{3}+o_f(1)\right)
                                      \log X\log\log X.
 \end{aligned}
\]
Take $c_f=\varepsilon_f/4>0$. We see that
$H\ge c_f\log X\log\log X$ for all sufficiently large $X$.
Choose $X_0(f)$ to satisfy this and the earlier thresholds, and
set $L=H$. For every real $X\ge X_0(f)$, we have constructed
the required interval in $(X/2,X]$, with all polynomial values
composite. This proves \eqref{eq:main} and
Theorem~\ref{thm:main}.
\end{proof}

\section{Acknowledgements}
This work is an output of a research project (HSE-BR-2025-84)
implemented as part of the Basic Research Program at HSE University.

\end{document}